\documentclass[12pt]{amsart} 
\usepackage[margin=1.2in]{geometry}
\usepackage{graphicx}
\usepackage{bm}
\usepackage{multirow}
\usepackage[table,xcdraw]{xcolor}
\usepackage{graphicx}
\usepackage[numbers]{natbib}
\usepackage{amsmath}
\usepackage{placeins}
\usepackage{amssymb}
\usepackage{epstopdf}
\usepackage{pdfpages}
\usepackage{amsthm}
\usepackage{xcolor}
\usepackage{setspace}
\usepackage{dsfont}

\usepackage{pgf,tikz}
\usepackage{pgfplots}
\usetikzlibrary{automata}
\usetikzlibrary{arrows} 
\usetikzlibrary{positioning}
\usetikzlibrary{snakes}

\newtheorem{theorem}{Theorem}[section]
\newtheorem{lemma}[theorem]{Lemma}

\newtheorem{corollary}[theorem]{Corollary}
\newtheorem{example}[]{Example}[section]
\newtheorem{definition}[theorem]{Definition}
\newtheorem{remark}[theorem]{Remark}

\newcommand*\E{\mathop{}\!\mathbb{E}}
\newcommand{\Exp}{\mathds{E}}

\newcommand{\ii}{\mbox{i}}
\newcommand{\Prob}{\mathop{}\!\mathbb{P}}

\newcommand{\vect}[1]{\vec{#1}}

\renewcommand{\vect}[1]{\boldsymbol{#1}}

\newcommand{\mat}[1]{\boldsymbol{#1}}

\newcommand*\TT{\mathop{}\!\bm{T}}
\newcommand*\T{\mathop{}\!\bm{t}}
\newcommand*\e{\mathop{}\!\bm{e}}

\newcommand*\p{\mathop{}\!\bm{\pi}}

\newcommand*\dd{\mathop{}\!\mathrm{d}}

\usepackage{multicol}

\author[H. \smash{Albrecher}]{Hansj\"org Albrecher}
\address[Hansj\"org Albrecher]{ Department of Actuarial Science, Faculty of Business and Economics and Swiss Finance Institute, University of Lausanne, CH-1015 Lausanne, Switzerland}

\email{{hansjoerg.albrecher@unil.ch}}

\author[M. \smash{Bladt}]{Martin Bladt}
\address[Martin Bladt]{Department of Actuarial Science, Faculty of Business and Economics, University of Lausanne, CH-1015 Lausanne, Switzerland}

\email{{martin.bladt@unil.ch}}

\author[M. \smash{Bladt}]{Mogens Bladt}
 \address[Mogens \smash{Bladt}]{Department of Mathematical Sciences, University of Copenhagen, Universitetsparken 5, DK-2100 Copenhagen \O, Denmark}
\email{bladt@math.ku.dk}

\title{Matrix Mittag--Leffler distributions and modeling heavy-tailed risks}

\begin{document}

\begin{abstract}
In this paper we define the class of matrix Mittag-Leffler distributions and study some of its properties. We show that it can be interpreted as a particular case of an inhomogeneous phase-type distribution with random scaling factor, and alternatively also as the absorption time of a semi-Markov process with Mittag-Leffler distributed interarrival times. We then identify this class and its power transforms as a remarkably parsimonious and versatile family for the modelling of heavy-tailed risks, which overcomes some disadvantages of other approaches like the problem of threshold selection in extreme value theory. We illustrate this point both on simulated data as well as on a set of real-life MTPL insurance data that were modeled differently in the past.
\end{abstract} 
\maketitle

\section{Introduction}
The modeling of heavy-tailed risks is a classical topic in probability, statistics and its applications to understand and interpret data, see e.g.\ Embrechts et al.\ \cite{ekt}, Beirlant et al.\ \cite{beirlant2006} and Klugman et al.\ \cite{klugman2012loss}. The folklore heavy-tailed distributions like Pareto, (heavy-tailed) Weibull and lognormal distributions can often serve as very useful benchmark models, particularly when only a few data points are available. In addition, the Pareto distribution is simple to work with and has an intuitive justification in terms of limit properties of extremes. However, in situations with more (but not an abundance of) data points, one often empirically observes that a simple Pareto distribution does not serve as a good model across the entire range of the distribution. This is also the case for more general parametric families like Burr or Benktander distributions. Traditionally, and according to a main paradigm of the extreme value statistics approach, this is handled by only using the largest available data points to estimate the tail behavior, and model the bulk of the distribution separately by another distribution, finally splicing together the respective parts (see e.g.\ Albrecher et al.\ \cite[Ch.4]{abt} for details). In insurance practice one often refers to this separate modeling as the modeling of \textit{attritional} and \textit{large} claims, and the resulting models are frequently referred to as \textit{composite models} \cite{pig}. A natural problem in this context is how to choose the threshold between the separate regions, often boiling down to the compromise of not leaving too few data points for the tail modeling. At the same time, the consequences  of that choice can be considerable, for instance for the determination of solvency capital requirements in insurance (cf. \cite[Ch.6]{abt} for illustrations). A considerable effort has therefore been made to develop techniques and criteria for an appropriate choice of such thresholds, see e.g.\ \cite{beirlant2006} for an overview and \cite{abbtrim} for a recent contribution in that direction. \\
If the confidence in the relevance of available data points for the description of the (future) risk is sufficiently high, another possible approach is to use a tractable, but much larger family of distributions and identify a good fit. A particularly popular candidate for such an approach is the class of phase-type (PH) distributions, see e.g.\ Asmussen et al.\ \cite{asmner}. The class of PH distributions is dense (in the sense of weak convergence) in the class of distributions on the positive real line, meaning that they can approximate any positive distribution arbitrarily well. They are, however, light-tailed, which may be a problem in applications which require a heavier tail and where the quantities of interest heavily depend on the tail behavior (as e.g.\ for ruin probabilities, cf. \cite{AsAl10}). Fitting heavy-tailed distributions with a PH distribution can then lead to requiring many phases (rendering its use computationally cumbersome), and the resulting model will still not capture the tail behavior in a satisfactory manner. Two approaches to remedy this problem are Bladt \& Rojas-Nandayapa \cite{BladtNandayapa2018} and 
Bladt et al.\ \cite{bladt-nielsen-samorodnitsky:2015}. 
In Albrecher \& Bladt \cite{ab18inh} recently another direction was suggested, namely to transform time in the construction of PH distributions (as absorption times of Markov jump processes), leading to inhomogeneous phase-type (IPH) distributions. For suitable transformations, this approach allows to transport the versatility of PH distributions into the domain of heavy-tailed distributions, by  introducing dense classes of genuinely heavy-tailed distributions. As a by-product, it was shown in \cite{ab18inh} that a class of matrix-Pareto distributions can be identified, where the scalar parameter of a classical Pareto distribution is replaced by a matrix, providing an intuitive and somewhat natural extension of the Pareto distribution, much as the matrix-exponential distribution, which is a powerful extension of the classical exponential distribution, see e.g.\ Bladt \& Nielsen \cite{bladt2017matrix}. \\

In this paper we establish another matrix version of a distribution, namely the Mittag-Leffler distribution (first studied by Pillai \cite{pillai}). While the identification of matrix versions of distributions is of mathematical interest in its own right, we will show that the resulting matrix Mittag-Leffler distributions (and its power transforms) have favorable properties for the modeling of heavy tails, and it can outperform some other modeling approaches in a remarkable way. {  Furthermore, we will identify this class of distributions as a particular extension of the IPH class, where the scaling is random. In addition, we will establish the matrix Mittag-Leffler distribution as the absorption time of a semi-Markov process with (scalar) Mittag-Leffler distributed inter-arrival times, extending the role of the exponential distribution for the inter-arrival in continuous-time Markov chains. }\\

The Mittag-Leffler function was first introduced in \cite{ml1904} and over the years turned out to be a crucial object in fractional calculus. It can be seen as playing the same role for fractional differential equations as the exponential function does for ordinary differential equations, see e.g.\ Gorenflo et al.\ \cite{gorenflo2014mittag} for a recent  overview. A recent application of fractional calculus for a particular risk model in insurance can be found in Constantinescu et al.\ \cite{const}. Mittag-Leffler functions with matrix argument were first introduced in Chikrii and Eidel’man \cite{chikrii2000generalized} and play a prominent role for identifying solutions of systems of fractional differential equations, see e.g.\ Garrappa and Popolizio \cite{Garrappa2018}. Here we will use them to define the class of matrix Mittag-Leffler (MML) distributions, which enjoy some attractive mathematical properties and are heavy-tailed with a regularly varying tail with index $\alpha<1$. While such extremely heavy tails with resulting infinite mean can be relevant in the modeling of operational risk \cite{chav06} and possibly insurance losses due to natural catastrophes \cite{abt}, in most applications of interest the tails are slightly less heavy. We therefore enlarge the class of matrix Mittag-Leffler distributions by also including its power transforms and estimate the corresponding power together with the other parameters from the data in the fitting procedure. The index of regular variation of this larger class of distributions can now be any positive number. Whereas the number of needed phases for a PH or IPH fit can be very large also due to multi--modality or other  irregularities in the shape of the main body of the distribution, we will see that the class of matrix Mittag-Leffler distribution and its power transforms (which we call \textit{power matrix Mittag-Leffler} (PMML) \textit{distributions}) offers a significant reduction in the number of phases needed to obtaining adequate fits. For this reason, it can even be worthwhile to scale light-tailed data points to heavy-tailed ones first, then apply a matrix Mittag-Leffler fit to the latter and transform the fit back to the original light-tailed scaling. This procedure is to some extent the reverse direction of the philosophy that underlied the PH fitting of heavy-tailed distributions. We will illustrate the potential advantage of this alternative approach in the numerical section at the end of the paper. \\

The remainder of the paper is organized as follows. Section \ref{sec2} recollects some useful definitions and properties of Mittag-Leffler functions, Mittag-Leffler distributions and PH distributions. Section \ref{sec3} then defines matrix Mittag-Leffler distributions and derives a number of its properties. We also give three explicit examples. Section \ref{sec4} establishes MML distributions as IPH distributions under a particular random scaling, which allows to intuitively understand the additional flexibility gained from using MML distributions for the modeling of heavy-tailed risks. {  We then also establish MML distributions as the absorption times of a semi-Markov process with ML distributed interarrival times, which is yet another perspective on the potential of the MML class as a modeling tool.} Finally, Section \ref{secillu} is devoted to the modeling of data using (power-transformed) MML distributions. We first illustrate the convincing performance of the numerical fitting procedure to something as involved as tri--modal data. Secondly, we consider an MTPL data set taken from \cite{abt} and already studied by various other means in the literature. We show that a plain maximum-likelihood fit to this data set gives a convincing fit to the entire range of the data with remarkably few parameters, and even identifies the tail index with striking accuracy when compared to recent extreme value techniques as in \cite{abbtrim}, without having to choose a threshold for the tail modeling at all. We then also provide an example where transforming light-tailed data into heavy-tailed ones, fitting with a PMML distribution and transforming back can lead to a much better fit for the same number of parameters than a classical phase-type  distribution. We then also discuss the signature of MML distributions in the tail in terms of the behavior of the Hill plot, which allows to develop an intuitive guess as to when MML distributions are particularly adequate for a fitting procedure of heavy tails. Finally, Section \ref{secconcl} concludes.

\section{Some relevant background}\label{sec2}
\subsection{Mittag--Leffler functions}
The  Mittag--Leffler (ML) function is defined by
\[ E_{\alpha, \beta}(z)=\sum_{k=0}^{\infty} \frac{z^{k}}{\Gamma(\alpha k+\beta)} ,\quad z\in\mathbb{C}, \]
where $\beta\in \mathbb{R}$  and $\alpha>0$. The ML function is an entire function if $\beta>0$, and it satisfies (see e.g.\ Erdelyi et al.\  \cite[p.210]{bateman1953higher}) 
\begin{equation*}
 \frac{\dd^m}{\dd z^m}\left[z^{\beta-1} E_{a, \beta}\left(z^{a}\right)\right]=z^{\beta-m-1} E_{\alpha, \beta-m}\left(z^{\alpha}\right) .
 \end{equation*} 
 This implies that (see \cite[Prop.2]{Garrappa2018})
\begin{equation}
  E_{\alpha, \beta}^
  {(k)}(z)=\frac{\dd^{k}}{\dd z^{k}} E_{\alpha, \beta}(z)=\frac{1}{\alpha^{k} z^{k}} \sum_{j=0}^{k} c_{j}^{(k)} E_{\alpha, \beta-j}(z) \label{eq:m-order-der-ML}
   \end{equation}
where
\[ c_{j}^{(k)}=\left\{\begin{array}{ll}{(1-\beta-\alpha(k-1)) c_{0}^{(k-1)}}, & {j=0}, \\
c_{j-1}^{(k-1)}+(1-\beta-\alpha(k-1)+j) c_{j}^{(k-1)}, & j=1, \ldots, k-1, \\ {1}, & {j=k} .
\end{array} \right.  \]
For a matrix $\vect{A}$, we may define its ML function as
\begin{align*}
E_{\alpha, \beta}(\vect{A})=\sum_{k=0}^{\infty} \frac{\vect{A}^{k}}{\Gamma(\alpha k+\beta)}.
\end{align*}
If $\beta>0$, one can then express the entire ML function of a matrix $\mat{A}$ by Cauchy's formula
\[ E_{\alpha,\beta}(\mat{A}) =\frac{1}{2\pi \ii}\int_\gamma E_{\alpha,\beta}(z)(z\mat{I}-\mat{A})^{-1}\dd z , \]
where $\gamma$ is a simple path enclosing the eigenvalues of $\mat{A}$. If $\mat{A}$ has a Jordan normal form $\mat{A}=\mat{P}\,\mbox{diag}(\mat{J}_1,...,\mat{J}_r)\,\mat{P}^{-1}$ with 
\[  \mat{J}_i = 
\begin{pmatrix}
  \lambda_i & 1 & 0 & \cdots & 0 \\
  0 & \lambda_i & 1 & \cdots & 0 \\
  0 & 0 & \lambda_i & \cdots & 0 \\
  \vdots & \vdots & \vdots & \vdots\vdots\vdots & \vdots \\
  0 & 0 & 0 & \cdots & \lambda_i
\end{pmatrix} ,
  \]
  then we may equivalently express $ E_{\alpha,\beta}(\mat{A})$ by
  \[   E_{\alpha,\beta}(\mat{A})=\mat{P}\, \mbox{diag}(E_{\alpha,\beta}(\mat{J}_1),...,E_{\alpha,\beta}(\mat{J}_r)) \,\mat{P}^{-1},  \]
  where
  \[  E_{\alpha,\beta}(\mat{J}_i) =  
\begin{pmatrix}
E_{\alpha,\beta}(\lambda_i) & E_{\alpha,\beta}^{(1)}(\lambda_i) & \frac{E_{\alpha,\beta}^{(2)}(\lambda_i)}{2!} & \cdots & \frac{E_{\alpha,\beta}^{(m_i-1)}(\lambda_i)}{(m_i-1)!} \\
0 & E_{\alpha,\beta}(\lambda_i) & E_{\alpha,\beta}^{(1)}(\lambda_i) & \cdots & \frac{E_{\alpha,\beta}^{(m_i-2)}(\lambda_i)}{(m_i-2)!} \\
0 & 0 & E_{\alpha,\beta}(\lambda_i) & \cdots & \frac{E_{\alpha,\beta}^{(m_i-3)}(\lambda_i)}{(m_i-3)!} \\
\vdots & \vdots & \cdots & \vdots\vdots\vdots & \vdots \\
0 & 0 & 0 & \cdots & E_{\alpha,\beta}(\lambda_i)
\end{pmatrix},
   \]
   and $m_i$ is the dimension of $\mat{J}_i$. 

   In either case, we shall need to evaluate the derivatives of ML functions at the eigenvalues, which by \eqref{eq:m-order-der-ML} implies the evaluation of ML functions with possibly negative indices  $\beta-j$. This is not a problem, but it is important that initially $\beta>0$ to ensure that the ML function is entire and thereby the existence of the Cauchy integral formula is guaranteed.  

For further properties on the ML function we refer e.g.\ to  \cite{bateman1953higher}, \cite{Garrappa2018}, \cite{Matychyn2018} and \cite{haubold}.

\subsection{Mittag--Leffler distributions}
A random variable having Mittag-Leffler (ML) distribution was defined in Pillai \cite{pillai} through the cumulative distribution function and consequently density given by
\begin{align*}
F_{\delta,\alpha}(x)&=1-E_{\alpha,1}(-(x/\delta)^{\alpha}),\quad x>0,\:\:0<\alpha\le 1,\\
f_{\delta,\alpha}(x)&
=\frac{x^{\alpha-1}}{\delta^\alpha}E_{\alpha,\alpha}(-(x/\delta)^\alpha),\quad x>0,\:\:0<\alpha\le 1,
\end{align*}
with Laplace transform
\begin{align}\label{ltrafo}
\frac{1}{1+(\delta u)^{\alpha}}.
\end{align}
A convenient representation, due to Kozubowski \cite{kozu01}, for a ML random variable $X$ is 
\begin{align*}
X\stackrel{d}{=}\delta Z R^{1/\alpha},
\end{align*}
where $Z$ is standard exponential and $R$ has cumulative distribution function
\begin{align*}
F_R(x)=\frac{2}{\pi \alpha}\left[\arctan\left(\frac{x}{\sin(\alpha\pi/2)}+\cot(\alpha\pi/2)\right)-\frac \pi 2 \right]+1.
\end{align*}
The tail behaviour of $R$ is equivalent to that of a Cauchy random variable, and hence $X$ is regularly varying with parameter $\alpha$ in the tail (see e.g.\ \cite[Prop.1.3.9]{mikosch1999regular}).

The following extension (for the case $\delta=1$) will also play a role in the sequel: a random variable $X$ is said to follow a generalized Mittag-Leffler (GML) distribution with parameters $\alpha$ $(0<\alpha\le 1)$ and $\beta>0$, if its Laplace transform is given by 
\[{\mathbb E}(e^{-uX})=(1+u^{\alpha})^{-\beta}.\]
The corresponding cumulative distribution function then is 
\begin{align*}
F_{\alpha,\beta}(x)&=\sum_{k=0}^\infty\frac{(-1)^k\Gamma(k+\beta)x^{\alpha(k+\beta)}}{\Gamma(\beta)\,k!\,\Gamma(1+\alpha(k+\beta))}=\sum_{k=0}^\infty\frac{(-1)^k x^{\alpha(k+\beta)}}{\mathfrak{B}(\beta,k)\,k\,\Gamma(1+\alpha(k+\beta))},
\end{align*}
where $\mathfrak{B}(x,y)$ is the Beta function (see Jose et al.\ \cite{jose2010generalized}). 
The analogous representation for a GML variable $X$ is
\begin{align}\label{prodgml}
X\stackrel{d}{=} W^{1/\alpha}S_\alpha,
\end{align}
where $W$ is Gamma with scale parameter 1 and shape parameter $\beta$, and $S_\alpha$ is a random variable with Laplace transform given by
\begin{align*}
{\mathbb E}(e^{-uS_\alpha})=\exp(-u^{\alpha}).
\end{align*}

{  \subsection{Phase--type distributions}\label{subsec:PH}
A random variable $\tau$ is said to be phase--type distributed with generator (or representation) $(\vect{\pi},\mat{T})$, and we write 
$\tau\sim \mbox{PH}(\vect{\pi},\mat{T})$, if it is the time until absorption of a (time--homogeneous) Markov jump process $\{ X_t\}_{t\geq 0}$
with state--space $E=\{1,2,...,p,p+1\}$ where states $1,...,p$ are transient and state $p+1$ is absorbing. The row vector $\vect{\pi}=(\pi_1,...,\pi_p)$ is the initial distribution, $\pi_i = \Prob (X_0=i)$, and $\mat{T}=\{ t_{ij}\}_{i,j=1,...,p}$ where $t_{ij}$ denotes the transition rates of jumps between transient states $i$ and $j$. We assume that $\pi_1+\cdots \pi_p+1$, i.e. $X_0$ cannot start in the absorbing state which would have caused an atom at zero. The intensity matrix for $\{ X_t\}_{t\geq 0}$ can be written as
\[ \mat{\Lambda} = \begin{pmatrix}
\mat{T} & \vect{t} \\
\vect{0} & 0 
\end{pmatrix} , \]
where $\vect{t}=(t_1,...,t_p)^\prime$ is a column vector of exit rates, i.e. $t_i$ is the rate of transition from state $i$ to the absorbing state $p+1$. We notice that $-\mat{T}\vect{e}=\mat{t}$, where $\vect{e}=(1,1,...,1)^\prime$ is the $p$--dimensional column vector of ones, since the row sums in $\mat{\Lambda}$ must all be zero. Hence $(\vect{\pi},\mat{T})$ fully parametrises the Markov process.  

The class of Phase--type distributions is dense in the class of distributions on the positive reals, meaning that they may approximate any positive distribution arbitrarily well. On the other hand, they constitute a class of probabilistically tractable distributions, which often allows for exact solutions to complex stochastic problems, and frequently in a closed form.
The theory is well developed with numerous applications in insurance risk and queueing theory (see e.g. \cite{bladt2017matrix} and reference therein). Phase--type distributions are light-tailed (i.e., their tail has an exponential decay), which makes them inadequate for modelling certain phenomena like insurance risks with heavy-tailed claims. 
 Recently, Albrecher \& Bladt \cite{ab18inh} proposed an extension of the PH construction principle to time--inhomogeneous Markov processes, in which case the absorption times can also be heavy-tailed with a wide spectrum of possible tail shapes. 

If $\tau\sim \mbox{PH}(\vect{\pi},\mat{T})$, then its density function is given by $f_\tau (x)=\vect{\pi}e^{\mat{T}x}\vect{t}$, its distribution function by $F_\tau (x)=1-\vect{\pi}e^{\mat{T}x}\vect{e}$ and its Laplace transform by $L_\tau (s) = \vect{\pi}(s\mat{I}-\mat{T})^{-1}\vect{t}$, where $\mat{I}$ denotes the identity matrix. The (fractional) moments are $\Exp (\tau^\alpha)=\Gamma (\alpha+1) \vect{\pi}(-\mat{T})^{-\alpha}\vect{e}$. For further details on Phase--type distributions we refer to \cite{bladt2017matrix}.

}


\section{Matrix Mittag--Leffler distributions}\label{sec3}
Let us now derive a matrix version of the Mittag-Leffler distribution by defining its Laplace transform and identifying the distribution associated to it. To this end, in view of  \eqref{ltrafo} consider the function
\begin{equation}
  \phi (u)=\vect{\pi}(u^\alpha \mat{I}- \mat{T})^{-1}\vect{t},  \quad 0<\alpha\le 1, \label{eq:LPT-ML-PH}
  \end{equation}
where $(\vect{\pi},\mat{T})$ is a PH generator. 
\begin{theorem}\label{th:LPT-MML}
$\phi (u)$ is the Laplace transform of a probability distribution.
\end{theorem}
\begin{proof}
Let $g(u)=u^\alpha$ and
\[  f(x)=\vect{\pi}(x\mat{I}-\mat{T})^{-1}\vect{t} . \]
Then $\phi (u)=f(g(u))$. Now $\mat{T}-u^\alpha\mat{I}$ is a sub--intensity matrix 
for all $u\geq 0$ and therefore $(u^\alpha\mat{I}-\mat{T})^{-1}$ is a non--negative matrix (Green matrix, see
\cite[p.134]{bladt2017matrix}). Thus 
\[  f^{(n)}(g(u))=(-1)^n n! \vect{\pi}\left( u^\alpha \mat{I}-\mat{T} \right)^{-n-1}\vect{t},   \]
which has sign $(-1)^n$. Concerning $g$,
\[  g^{(j)}(u) = \alpha (\alpha-1)\cdots (\alpha-j+1)u^{\alpha -j} \]
has sign $(-1)^{j+1}$.
We shall employ Fa\'a di Bruno's formula,
\[ \frac{\dd ^{n}}{\dd x^{n}} f(g(x))=\sum \frac{n !}{m_{1} ! m_{2} ! \cdots m_{n} !} \cdot f^{\left(m_{1}+\cdots+m_{n}\right)}(g(x)) \cdot \prod_{j=1}^{n}\left(\frac{g^{(j)}(x)}{j !}\right)^{m_{j}}, \]
where the summation is over $n$--tuples for which
\[ 1 \cdot m_{1}+2 \cdot m_{2}+3 \cdot m_{3}+\cdots+n \cdot m_{n}=n , \]
to determine the sign of $\phi^{(n)}(u)$. Notice that
\begin{eqnarray*}
\mbox{sign}\left( f^{(m_1+m_2+\cdots + m_n)}(g(u)) \prod_{i=1}^n g^{(i)}(u)^{m_i}  \right)&=&(-1)^{m_1+\cdots + m_n} \prod_{i=1}^n (-1)^{m_i (i+1)} \\
&=&(-1)^{2\sum_i m_i} (-1)^{\sum_i i m_i} \\
&=&(-1)^n .
\end{eqnarray*}
Hence all terms in the summation have the same sign $(-1)^n$, 
and we conclude that also the sum itself has sign $(-1)^n$, i.e.  $\mbox{sign}(\phi^{(n)}(u))=(-1)^n$. Since $\phi (0)=\vect{\pi}(-\mat{T})\vect{t}=\vect{\pi}(-\mat{T})(-\mat{T}\vect{e})=\vect{\pi}\vect{e}=1$, the result then follows with Bernstein's theorem (see \cite[p.439]{Feller-vol-2}).
\end{proof}
\begin{remark}\normalfont
Note that the proof does not rely on the special form of $u^\alpha$, and it follows that if $g$ is
any function with $g(0)=0$ and $-g$ completely monotone, then 
\[
  \phi (u)=\vect{\pi}(g(u) \mat{I}- \mat{T})^{-1}\vect{t}     \]
  is the Laplace transform of a probability distribution, which may also be useful in other contexts.
\end{remark}

{ 
\begin{theorem}\label{th:density-MML}
Let $X$ be a random variable with Laplace transform \eqref{eq:LPT-ML-PH}. Then the density function of $X$ is given by 
\[  f(x)=x^{\alpha-1}\vect{\pi}\,E_{\alpha,\alpha}\left(\mat{T}x^\alpha \right)\,\vect{t}  .  \]
\end{theorem}}
\begin{proof}
We show that $f$ has the required Laplace transform. For $u$ sufficiently large one has
\begin{eqnarray*}
\int_0^\infty e^{-ux}f(x)\dd x&=& \vect{\pi}\int_0^\infty e^{-ux}\sum_{n=0}^\infty \frac{\mat{T}^n x^{\alpha n}}{\Gamma ((n+1)\alpha)}x^{\alpha -1}\dd x \ \vect{t} \\
&=&\vect{\pi}\sum_{n=0}^\infty \frac{\mat{T}^n}{\Gamma ((n+1)\alpha)}\int_0^\infty x^{\alpha (n+1)-1}e^{-ux}\dd x\ \vect{t}\\
&=&\vect{\pi}\sum_{n=0}^\infty \frac{\mat{T}^n}{\Gamma ((n+1)\alpha)}\Gamma ((n+1)\alpha)u^{-(n+1)\alpha} \vect{t} \\
&=&-\vect{\pi}\sum_{n=0}^\infty \mat{T}^{n+1}u^{-(n+1)\alpha} \vect{e} \\
&=&-\vect{\pi}\sum_{n=1}^{\infty} (-\mat{T}u^{-\alpha})^n \vect{e} \\
&=&-\vect{\pi}(\mat{I}-\mat{T}u^{-\alpha})^{-1}(\mat{T}u^{-\alpha})\vect{e} \\
&=&\vect{\pi}(u^\alpha \mat{I}-\mat{T})^{-1}\vect{t} .
\end{eqnarray*}
The result for all $u\in [0,\infty)$ then follows by analytic continuation.
\end{proof}

\begin{definition}
Let $(\vect{\pi},\mat{T})$ be a PH generator and let $0<\alpha\leq 1$. A random variable $X$ is said to have a matrix Mittag--Leffler distribution, if its Laplace transform is given by
\[  \E (e^{-uX})=\vect{\pi}(u^\alpha \mat{I}- \mat{T})^{-1}\vect{t} . \]
In this case we write $ X\sim \mbox{MML}(\alpha,\vect{\pi},\mat{T})$.
\end{definition}
{ 
\begin{remark}\rm 
The proof Theorem \ref{th:density-MML} shows that for any triplet $(\vect{\pi},\mat{T},\vect{t})$ for which
 \begin{equation}\label{medens} f(x) = \vect{\pi}e^{\mat{T}x}\vect{t} \end{equation}
is a density function, the function $\phi (u) =\vect{\pi}(u^\alpha \mat{I}- \mat{T})^{-1}\vect{t} $ is indeed a
Laplace transform of a probability distribution. Distributions with density \eqref{medens} are referred to as matrix--exponential distributions, and contain the class of phase--type distributions as a strict subset. We have stated the above definition in terms of a phase--type generator, but it is hence clear that the construction also applies to any matrix--exponential distribution. While most results in the following could be stated in terms of a matrix--exponential triplet, for the sake of simplicity and notation we restrict the parameters to phase--type generators only. 
\end{remark}
}

\begin{corollary}\label{cor:cdf-MML}
Let $X\sim \mbox{MML}(\alpha,\vect{\pi},\mat{T})$. Then the cumulative distribution function for $X$ is given by  
\[  F(x)=1-\vect{\pi}E_{\alpha,1}\left(\mat{T}x^\alpha \right)\vect{e}  .  \]
\end{corollary}
\begin{proof}
The derivative of this function is
\begin{eqnarray*}
F^\prime(x)&=&-\vect{\pi}\sum_{n=1}^\infty \frac{\mat{T}^n n x^{\alpha (n-1)}\alpha x^{\alpha-1}}{\Gamma (1+\alpha n)}\vect{e} \\
&=&-x^{\alpha-1}\vect{\pi}\sum_{n=1}^\infty \frac{ (\mat{T}x^\alpha)^{n-1}}{\Gamma (\alpha n)}\mat{T} \vect{e} \\
&=&x^{\alpha-1}\vect{\pi}E_{\alpha,\alpha}(\mat{T}x^\alpha)\vect{t} ,
\end{eqnarray*}
which indeed coincides with the density. It remains to note that $F$ satisfies the boundary condition $F(0)=0$.
\end{proof}
One can now realize that the matrix ML distribution provides an extension of representation \eqref{prodgml} in the following way:
\begin{theorem}\label{repPHstable}
Let $X\sim \mbox{MML}(\alpha,\vect{\pi},\mat{T})$. Then
\begin{align}\label{prodrepresentation}
X\stackrel{d}{=}W^{1/\alpha}S_\alpha,
\end{align}
where $W\sim \mbox{PH}(\vect{\pi},\:\vect{T})$, and $S_\alpha$ is an independent (positive stable) random variable with Laplace transform given by $\exp(-u^{\alpha})$.
\end{theorem}
\begin{proof}
Simply note that
\begin{align*}
\E(\exp(-u W^{1/\alpha}S_\alpha))&=\int_0^\infty \E(\exp(-ux^{1/\alpha}S_\alpha))\vect{\pi} e^{\vect{T}x}\vect{t} \dd x\\
&=\vect{\pi} \int_0^\infty e^{-u^{\alpha}x} e^{\vect{T}x}\dd x\, \vect{t}\\
&=\vect{\pi} \int_0^\infty e^{-x(u^{\alpha}\vect{I}-\vect{T})}\dd x\, \vect{t}\\
&=\vect{\pi}(u^\alpha \mat{I}- \mat{T})^{-1}\,\vect{t}.
\end{align*}
\end{proof}
\begin{corollary}
Let $X\sim \mbox{MML}(\alpha,\vect{\pi},\mat{T})$. The fractional moments of order $\rho<\alpha\le 1$ are given by
\begin{align*}
\E(X^\rho)=\frac{\Gamma(1-\rho/\alpha)\Gamma(1+\rho/\alpha) \boldsymbol{\pi}(-\boldsymbol{T})^{-\rho/\alpha} \boldsymbol{e}}{\Gamma(1-\rho)}.
\end{align*}
\end{corollary}
\begin{proof}
It is known \cite{Wolfe1975OnMO} that the fractional moments of a random variable $S_\alpha$ with Laplace transform $\exp(-u^{\alpha})$ are given by
\begin{align*}
\E(S_\alpha^\rho)=\frac{\Gamma(1-\rho/\alpha)}{\Gamma(1-\rho)}.
\end{align*}
From \cite{bladt2017matrix} we know that the $\nu$th fractional moment of a random variable $W$ with  PH($\vect{\pi},\:\vect{T}$) distribution is given by

\[\E(W^\nu)= \Gamma(\nu+1) \boldsymbol{\pi}(-\boldsymbol{T})^{-\nu} \boldsymbol{e}. \]
By Theorem \ref{repPHstable}, and setting $\nu=\rho/\alpha$, $X$ will have the $\rho$th fractional moment given by
\begin{align*}
\E(X^\rho)=\E(S_\alpha^{\rho})\E(W^{\rho/\alpha})=\frac{\Gamma(1-\rho/\alpha)\Gamma(1+\rho/\alpha) \boldsymbol{\pi}(-\boldsymbol{T})^{-\rho/\alpha} \boldsymbol{e}}{\Gamma(1-\rho)}.
\end{align*}
\end{proof}
\begin{remark}\normalfont
Representation \eqref{prodrepresentation} is not only useful for establshing closed-form formulas, but it also suggests a simple and efficient simulation technique for $X$. Simulation algorithms for PH and stable distributions are for instance available in the statistical software R via the packages \texttt{actuar} and \texttt{stabledist}, respectively.
\end{remark}

\begin{example}\label{ex:erlang}\rm
Consider $X\sim \mbox{MML}(\alpha,\vect{\pi},\mat{T})$ where $(\vect{\pi},\mat{T})$ is the PH representation of an Erlang distribution with $p$ ($p\in{\mathbb N}$) stages and intensity $\lambda$, i.e.
$\vect{\pi}=(1,0,...,0)$ and 
\[ \mat{T} = \begin{pmatrix}
-\lambda & \lambda & 0 & ... & 0 & 0 \\
0 & -\lambda & \lambda & ... & 0 & 0 \\
0 & 0 & -\lambda & ... & 0 & 0 \\
\vdots & \vdots &\vdots  & \vdots \vdots \vdots  &\vdots &\vdots  \\
0 & 0 & 0 & ... & -\lambda & \lambda \\
0 & 0 & 0 & ... & 0 & -\lambda
\end{pmatrix} .\]
In this case
\[ (s\mat{I}-\mat{T})^{-1} = 
\begin{pmatrix}
\frac{1}{s+\lambda} & \frac{\lambda}{(s+\lambda)^2} &  \frac{\lambda^2}{(s+\lambda)^3} & ... &  \frac{\lambda^{p-1}}{(s+\lambda)^p} \\
0 & \frac{1}{s+\lambda} & \frac{\lambda}{(s+\lambda)^2} &  ... & \frac{\lambda^{p-2}}{(s+\lambda)^{p-1}} \\
0 & 0 &  \frac{1}{s+\lambda} & ... & \frac{\lambda^{p-3}}{(s+\lambda)^{p-2}} \\
\vdots & \vdots & \vdots & \vdots\vdots\vdots & \vdots \\
0 & 0 & 0 & ... & \frac{1}{s+\lambda}
\end{pmatrix}, \]
so
\begin{eqnarray*}
f(x)&=&x^{\alpha-1}\vect{\pi}E_{\alpha,\alpha}(\mat{T}x^{\alpha})\vect{t} \\
&=&x^{\alpha-1}\frac{1}{2\pi \ii}\int_\gamma E_{\alpha,\alpha}(s)\vect{\pi}(s\mat{I}-x^\alpha \mat{T})^{-1}\,\vect{t}\,\dd s  \\
&=&x^{\alpha-1}\frac{1}{2\pi \ii}\int_\gamma  E_{\alpha,\alpha}(s) \lambda \frac{(\lambda x^\alpha)^{p-1}}{(s+x^\alpha \lambda)^p} \dd s \\
&=&\frac{\lambda^p x^{\alpha p-1}}{(p-1)!} E_{\alpha,\alpha}^{(p-1)}(-\lambda x^\alpha),
\end{eqnarray*}
where $\gamma$ is a simple path enclosing $-\lambda x^\alpha$, and where in the last step we used the residue theorem. Note that this corresponds to the GML random variable given already for general shape parameter $p=\beta\in{\mathbb R}_+$ in \eqref{prodgml} (although this explicit form of the density was not given in \cite{jose2010generalized}). Figure \ref{erlangfig} depicts the density for several choices of parameters. {  The parameter $\alpha$ controls the heaviness of the tail (and more generally the deviation from exponentiality of the Mittag-Leffler function), the parameter $p$ determines the shape of the body of the distribution (since larger $p$ implies a more pronounced Erlang component), and $\lambda$ is a scaling parameter.}
\begin{figure}[h]
\centering
\includegraphics[width=8cm,trim=.5cm .5cm .5cm .5cm,clip]{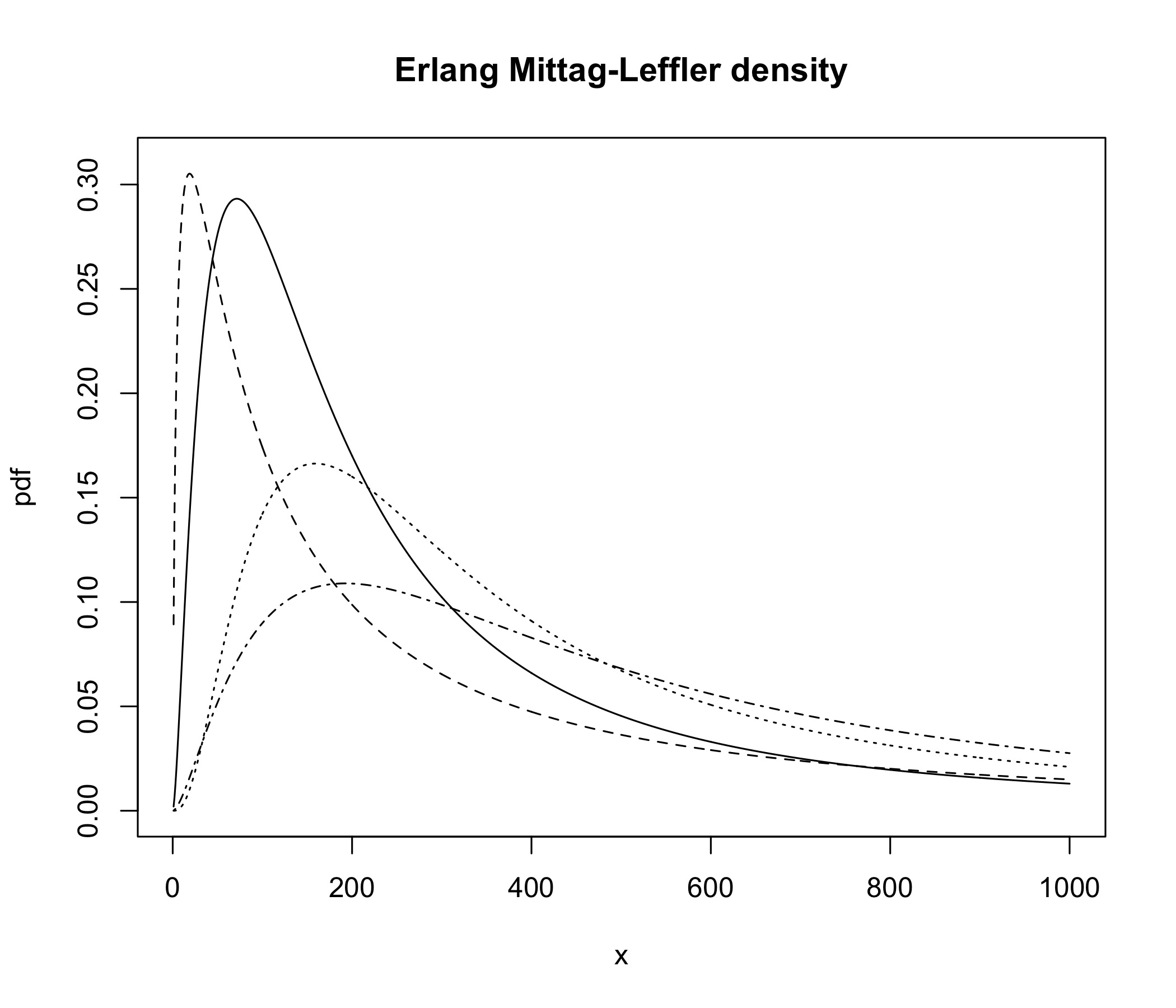}
\caption{Density of a $\mbox{MML}(\alpha,\vect{\pi},\mat{T})$, for $\alpha=0.7$ and Erlang phase-type component with $p=4$ and $\lambda=2$ (solid). The remaining curves have a change in one of the parameters: $\alpha=0.5$ (dashed), $p=6$ (dotted), and $\lambda=1$ (dashed and dotted).} 
\label{erlangfig}
\end{figure}
\qed 
\end{example}

\begin{example}\rm
Mixture of Erlang distributions form the simplest sub--class of PH distributions which are dense in the class of distributions on the positive real line (in the sense of weak convergence). Let $h$ be the density
\[ h(x)=\sum_{i=1}^m \theta_i f_{{Erl}}(x;{p_i},\lambda_i) , \]
where $f_{{Erl}}(x;{p_i},\lambda_i) = \lambda_i^{p_i}x^{p_i-1}\exp (-\lambda_i x)/(p_i-1)!$ denotes the density of an Erlang distribution and $\theta_i\geq 0$ are weights with $\sum\theta_i=1$. Then it follows immediately from Example \ref{ex:erlang} that the corresponding MML distribution has density
\[  f(x) = \sum_{i=1}^m \theta_i \frac{\lambda_i^{p_i} x^{\alpha p_i-1}}{(p_i-1)!} E_{\alpha,\alpha}^{(p_i-1)}(-\lambda_i x^\alpha). \]
In Figure \ref{mixerlangfig}, a trimodal distribution is considered, corresponding to $X\sim\mbox{MML}(\alpha,\vect{\pi},\mat{T})$, for a mixture of three Erlang PH components. The densities of  $\log(X)$ and $X$ are both depicted. 
\qed 
\end{example}

\begin{figure}[hh]
\centering
\includegraphics[width=14cm,trim=.5cm 2cm 2cm 2cm,clip]{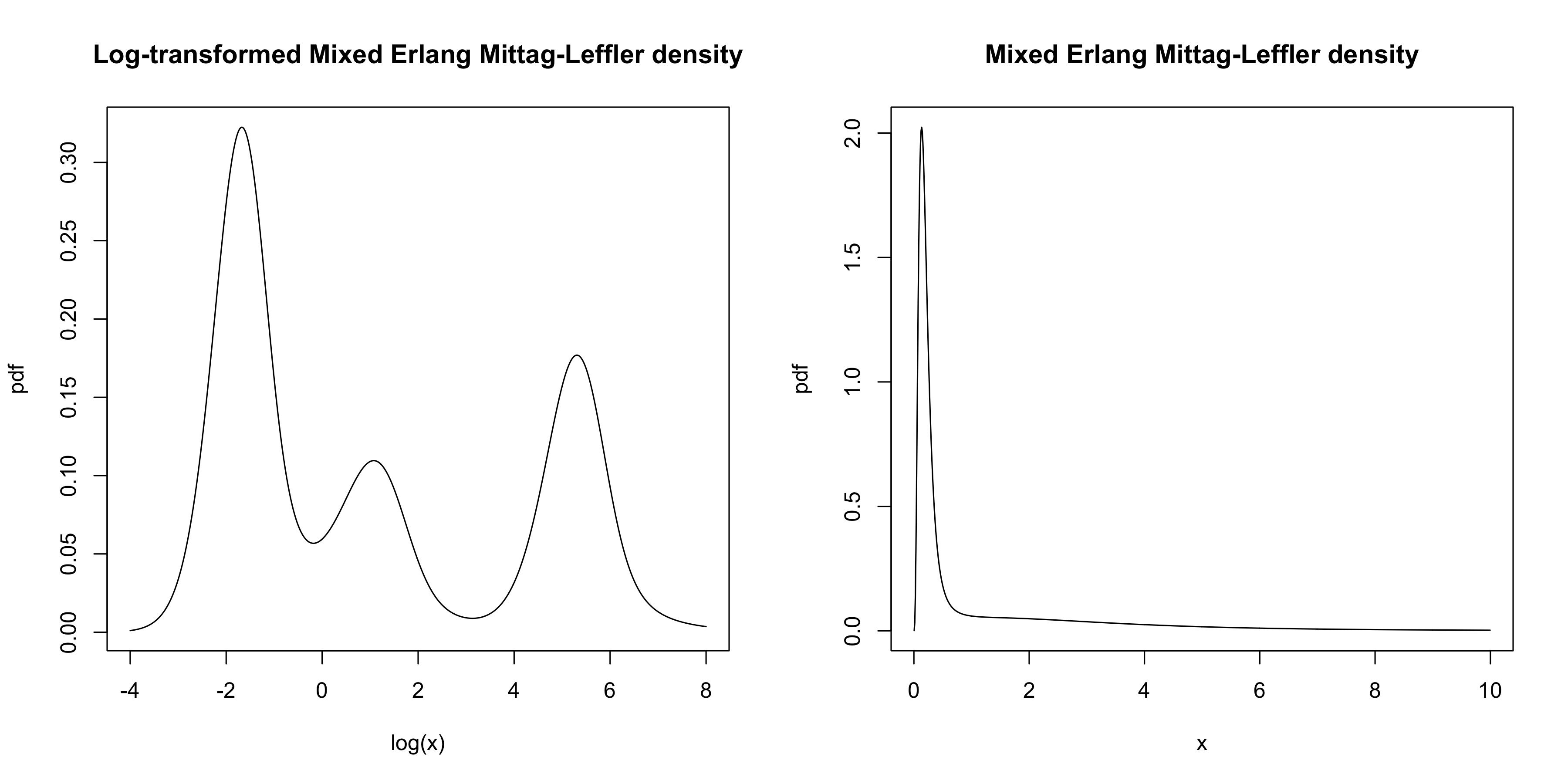}
\caption{Densities of $\log(X)$ and of $X$, where $X\sim\mbox{MML}(\alpha,\vect{\pi},\mat{T})$ is a mixture of $m=3$ Erlang PH components with parameters $\alpha=0.9$, $p_1=5,\:p_2=3,\:p_3=4$, $\lambda_1=20,\:\lambda_2=1,\:\lambda_3=0.03$,  $\theta_1=0.5,\:\theta_2=0.2,\:\theta_3=0.3$.} 
\label{mixerlangfig}
\end{figure}

\begin{example}\rm
A Coxian phase--type distribution has a representation of the form
\[  \vect{\pi}=(\pi_1,...,\pi_p), \ \ \mat{T}=
\begin{pmatrix}
-\lambda_1 & \lambda_1 & 0 & ... & 0 \\
0 & -\lambda_2 & \lambda_2  & ... & 0 \\
0 & 0 & -\lambda_3 & ... & 0 \\
\vdots & \vdots & \vdots & \vdots\vdots\vdots & \vdots\\
0 & 0 & 0 & ... & -\lambda_p
\end{pmatrix} ,
   \]
   where all $\lambda_i$, $i=1,\dots,p,$ are distinct.
   In this case
   \[  (z\mat{I}-\mat{T})^{-1} =
\begin{pmatrix}
\frac{1}{z+\lambda_1} & \frac{\lambda_1}{(z+\lambda_1)(z+\lambda_2)} & \frac{\lambda_1\lambda_2}{(z+\lambda_1)(z+\lambda_2)(z+\lambda_3)} & \cdots & \frac{\lambda_1\cdots \lambda_{p-1}}{(z+\lambda_1)(z+\lambda_2)\cdots (z+\lambda_p)} \\
0 & \frac{1}{z+\lambda_2} & \frac{\lambda_2}{(z+\lambda_2)(z+\lambda_3)} & \cdots & \frac{\lambda_2\cdots \lambda_{p-1}}{(z+\lambda_2)(z+\lambda_3)\cdots (z+\lambda_p)} \\
0 & 0 & \frac{1}{z+\lambda_3} & \cdots & \frac{\lambda_3\cdots \lambda_{p-1}}{(z+\lambda_3)(z+\lambda_4)\cdots (z+\lambda_p)} \\
\vdots & \vdots & \vdots & \vdots\vdots\vdots & \vdots\\
0 & 0 & 0 & \cdots & \frac{1}{z+\lambda_p}
\end{pmatrix},
     \]
 so
 \begin{eqnarray*}
f(x)
&=&x^{\alpha-1}\frac{1}{2\pi \ii}\int_\gamma E_{\alpha,\alpha}(s)\vect{\pi}(s\mat{I}-x^\alpha \mat{T})^{-1}\vect{t}\dd s  \\
&=& x^{\alpha-1}\sum_{j=1}^p \pi_j \frac{1}{2\pi \ii}\int_\gamma E_{\alpha,\alpha}(s) \frac{(\lambda_jx^{\alpha})(\lambda_{j+1}x^{\alpha})\cdots (\lambda_{p-1}x^{\alpha})\lambda_p}{(s+\lambda_jx^\alpha)(s+\lambda_{j+1}x^\alpha)\cdots (s+\lambda_px^\alpha)}\dd s \\
&=&\sum_{j=1}^p \pi_j x^{\alpha (p-j+1)-1}\left(\prod_{k=j}^{p}\lambda_k \right)\frac{1}{2\pi \ii}\int_\gamma  \frac{E_{\alpha,\alpha}(s)}{(s+x^\alpha\lambda_j)\cdots (s+x^\alpha\lambda_p)}\dd s \\
&=&\sum_{j=1}^p \pi_j x^{\alpha (p-j+1)-1}\left(\prod_{k=j}^{p}\lambda_k \right)\sum_{m=j}^p  \frac{E_{\alpha,\alpha}(-\lambda_m x^\alpha)}{\displaystyle\prod_{\stackrel{n=j}{n\neq m}}^p (-x^\alpha \lambda_m+x^\alpha \lambda_n)}\\
&=& x^{\alpha-1}\sum_{j=1}^p \pi_j  \left(\prod_{k=j}^{p}\lambda_k \right)
\sum_{m=j}^p \frac{E_{\alpha,\alpha}(-\lambda_m x^\alpha)}{\displaystyle\prod_{\stackrel{n=j}{n\neq m}}^p ( \lambda_n- \lambda_m)}.
\end{eqnarray*}  
In Figure \ref{coxfig} four such densities are plotted.  
\begin{figure}[hh]
\centering
\includegraphics[width=8cm,trim=.5cm .5cm .5cm .5cm,clip]{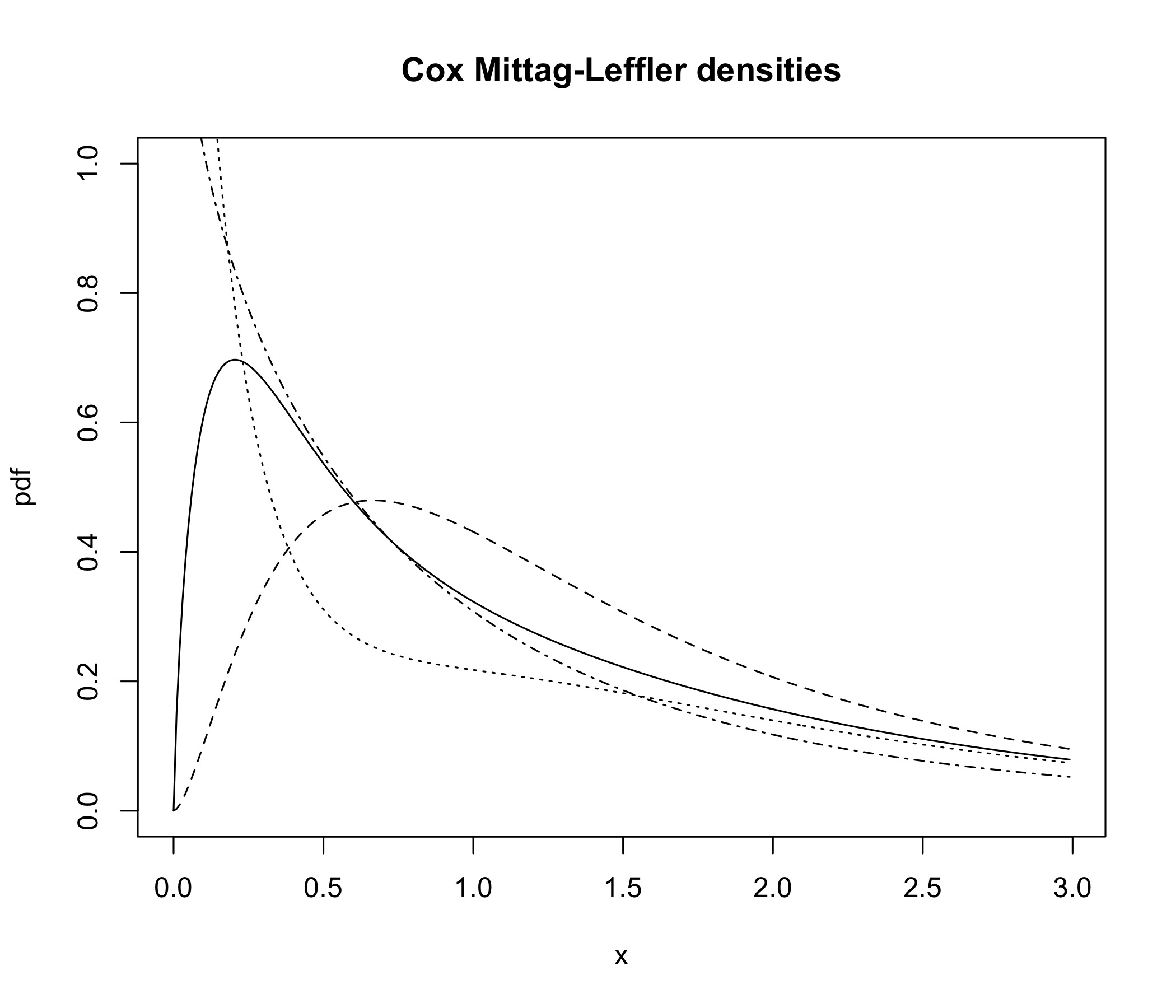}
\caption{Density of $X\sim\mbox{MML}(\alpha,\vect{\pi},\mat{T})$, with Coxian PH component ($p=4$, $\alpha=0.9$, $\lambda_1=1,\:\lambda_2=2,\:\lambda_3=3, \:\lambda_4=4$, and $\vect{\pi}=(0.5,0,0.5,0)$ (solid), $\vect{\pi}=(0.5,0.5,0,0)$ (dashed), $\vect{\pi}=(0.5,0,0,0.5)$ (dotted), and $\vect{\pi}=(0.25,0.25,0.25,0.25)$ (dashed--dotted)).} 
\label{coxfig}
\end{figure}
   \qed
\end{example}

\section{Sample path respresentations}\label{sec4}
{  We now provide two different representations of the MML distribution as sample path properties of a stochastic process. The first one will be as an absorption time of a randomly scaled time-inhomogeneous Markov jump process, and the second one as an absorption time of a particular semi-Markov process (where random time scaling is not needed).}

\subsection{Random time-inhomogeneous phase--type distributions}\label{subsec:random-scaling}
{  We recall from Section \ref{subsec:PH} that a
 random variable is PH distributed if it is the time until absorption of a time--homogeneous Markov jump process on a finite state--space, where one state is absorbing and the remaining states are transient. In this section we show that 
 a MML distribution can be interpreted as a  time--inhomogeneous phase--type distribution with random intensity matrix.}
 
 Let us define a random time-inhomogeneous Markov jump process $X_t$ as a jump process with state space $E=\{1,2,\dots,p,p+1\}$, where $p+1$ is an absorbing state, the remaining states being transient, and the intensity matrix given by
\begin{align}\label{intmat}
\bm{\Lambda}(t)=\frac{1}{Y}\begin{pmatrix}
    \bm{T}(t) & \bm{t}(t)  \\
    \bm{0} & 0
  \end{pmatrix},
\end{align}
where $\bm{t}(t)=-\bm{T}(t)\bm{e}$, $\e=(1,1,\dots,1)^T$, $\bm{0}=(0,0,\dots,0)$, $\p=(\pi_1,\dots,\pi_p)$,
\begin{align*}
\Prob(X_0=p+1)=0,\quad \Prob(X_0=k)=\pi_k, \:\: k=1,\dots,p,
\end{align*}
and the independent, positive random variable $Y$ is a random scaling factor. For the resulting absorption time
\begin{align*}
\tau=\inf\{t\ge0\,:X_t=p+1\},
\end{align*}
we write $\tau\sim \mbox{RIPH}(Y,\p,\TT(t))$. Note that the special case $Y\equiv 1$ corresponds to the IPH class in \cite{ab18inh}. 

If furthermore we can write $\TT(t)=\lambda(t)\TT$, the intensity matrix \eqref{intmat} takes the form \begin{align*}
\bm{\Lambda}(t)=\frac{\lambda(t)}{Y}\begin{pmatrix}
    \bm{T} & \bm{t}  \\
    \bm{0} & 0
  \end{pmatrix},
\end{align*}
in which case we write $\tau\sim \text{RIPH}(Y,\p,\TT,\lambda)$. The following result is then immediate.

\begin{theorem}\label{prodrep}

Let $\tau\sim \text{RIPH}(Y,\p,\TT,\lambda)$ for a random variable $Y$ with density $g$.  Then the density $f$ and distribution function $F$ of $\tau$ are given by
\begin{align*}
f(x)&=\int_0^\infty \lambda(x/v)\p\exp\left(\int_0^{x/v}\lambda(u)\dd u\TT\right)\T\,\frac{g(v)}{v}\dd v,\\
F(x)&=1-\int_0^\infty \p\exp\left(\int_0^{x/v}\lambda(u)\dd u \TT\right)\e \,g(v)\dd v.
\end{align*}
Furthermore, if $\lambda(t)$ is a strictly positive function and we define $h$ by
\begin{align*}
h^{-1}(x)=\int_0^x \lambda(t)\dd t,
\end{align*}
then
\begin{align}\label{prodrepform}
\tau\stackrel{d}{=}h(\tau_0)\cdot Y,
\end{align}
where $\tau_0\sim PH(\p,\TT)$.
\end{theorem}

Combining Theorem \ref{repPHstable} and Theorem \ref{prodrep}, the matrix Mittag-Leffler random variable $X\sim \mbox{MML}(\alpha,\vect{\pi},\mat{T})$ can hence be interpreted as a particular random scaling of a time-inhomogeneous phase-type distribution $\tau_0\sim PH(\p,\TT)$ with $h(x)=x^{1/\alpha}$ (translating into $\lambda(t)=\alpha t^{\alpha-1}$) and heavy-tailed random scaling factor $Y=S_{\alpha}$. As we will illustrate in Section \ref{secillu}, this represents a particularly versatile yet simple class of random variables for fitting real data. 

\begin{remark}\normalfont
	For any random variable $W$ with cumulative distribution function $F_W$, it is possible to write
	\begin{align*}
	W\stackrel{d}{=}F_W^{-1}(1-\exp(-E))=:h_W(E),
	\end{align*}
	where $E$ is a unit mean exponential random variable. For modeling purposes, the rationale in \cite{ab18inh} can be interpreted as approximating the transformation function $h_W$ before-hand by some function $h$ (in absence of the knowledge of $W$) and then replacing $E$ with a general PH distribution, providing flexibility for the fit with often explicit formulas for the resulting random variable. The matrix-Pareto distributions defined in \cite{ab18inh} are then the special case $Y\equiv 1$ and (up to a constant) $h(x)=e^x-1$ in \eqref{prodrepform}, which entails $\lambda (t)=1/(1+t)$. 
	The fitting in that case is particularly parsimonious for distributions 'close' to a Pareto distribution (where the distance concept here is then inherited from the distance in the PH domain after the log-transform). The general representation \eqref{prodrepform}, in contrast, allows to introduce a potential heavy-tail behavior also through the random scaling factor $Y$, providing more flexibility for the shape of the function $h(x)=x^{1/\alpha}$ (through the choice of $\alpha$) in a fitting procedure while keeping the resulting expressions tractable. 
\end{remark}

\begin{remark}\normalfont
{  The form \eqref{prodrepform} may also suggest to consider -- for modelling purposes -- the somewhat simpler case of a PH variable multiplied by a standard Pareto variable with tail index $\beta>0$, that is $h(x)=x$ and $f_Y(y)=\beta x^{-\beta-1}$, $x\ge 1$. For general $\tau_0\sim PH(\p,\TT)$ it is straightforward to see that then
\begin{align*}
f_\tau(t)&=\beta z^{-\beta-1}\int_0^tw^{\beta}\pi \exp(\mat{T}w)\vect{t} \dd w\\
&=\beta z^{-\beta-1} m_\beta F_{m_\beta}(t),
\end{align*}
where $m_\beta$ is the $\beta$-th moment of $\tau_0$, and $F_{m_\beta}$ is its $\beta$-th moment distribution.

For $\tau_0\sim \text{Exp}(\lambda)$ this simplifies to
$$f_\tau(t)=\frac{\beta (\lambda t)^{-\beta/2}\exp(-\lambda t/2) \mathcal{W}_M\left(\frac \beta 2,\frac \beta 2 +\frac 12,\lambda t\right) }{\lambda t(\beta+1)},$$
where
$$\mathcal{W}_M(k,m,z)=z^{m+1 / 2} e^{-z / 2} \sum_{n=0}^{\infty} \frac{\left(m-k+\frac{1}{2}\right)_{n}}{n !(2 m+1)_{n}} z^{n}$$
 is the Whittaker M function and $(x)_n$ is the Pochhammer symbol. Inserting a matrix into the third argument of this function, one may now proceed again with the matrix version of Cauchy's formula and by Jordan decomposition, potentially giving rise to a theory similar to the one for Mittag-Leffler distributions. However, this direction is not the focus of the present paper. In addition, a key difference between the above product construction and the MML distribution will be discussed in Section \ref{subsm} below in the context of a non-random path representation, for which the fine properties of the ML distribution play a crucial role.} 
\end{remark}

{  \subsection{Semi--Markov framework}\label{subsm}
 Let $E^\ast=\{1,2,...,p\}$ be a state space  and
let $\mat{Q}=\{ q_{ij} \}_{i,j\in E}$ denote a transition matrix for some Markov chain $\{ Y_n \}_{n\in \mathbb{N}}$ defined on $E^\ast$. We assume that $q_{ii}=0$ for all $i$. $\{ Y_n \}_{n\in \mathbb{N}}$ will serve as an embedded Markov chain in a Markov renewal process. 

Let $\alpha \in (0,1]$ and $\lambda_i>0$. For $i=1,...,p,$ let 
$T^i_n$ be i.i.d.\ random variables with a Mittag--Leffler distribution $\mbox{ML}(\alpha,\lambda_i)$, where we use the parametrisation such that the density of a generic $T^i$ is given by
  \begin{equation}
  f_{i}(x)=\lambda_i x^{\alpha-1} {E}_{\alpha,\alpha}(-\lambda_i x^{\alpha}) . \label{def:ML-inter-arrivals}
   \end{equation}
An alternative common parametrisation is obtained in terms of the parameter $\rho_i$ that satisfies  $\rho_i^{-\alpha} = \lambda_i$. 

We now construct a semi--Markov process $\{ X_t \}_{t\geq 0}$ as follows. Let $S_0=0$ and
\[ S_n = \sum_{i=1}^n T^{Y_i}_i , \ \ n\geq 1 .\]
Define
\begin{equation}
  X_t = \sum_{n=1}^\infty Y_{n-1} 1_{\{ S_{n-1}\leq t <S_n  \}}  . \label{def:X-process}
\end{equation}
Then $\{ X_t\}_{t\geq 0}$ changes states according to the Markov chain $Y_n$, $S_n$ denotes the time of the $n$'th jump, and the sojourn times in states $i$ are Mittag--Leffler distributed with paramters $(\alpha,\lambda_i)$. The construction is illustrated in Figure \ref{fig:X-process}.

\begin{figure}[h]
\begin{tikzpicture}[scale=0.75,domain=-1:14]
\draw[->] (0,0)--(14.0,0) node[right] {$t$};
\draw[-] (0,0) --(0,2);
\draw[-,dashed] (0,2)--(0,3);
\draw[->] (0,3)--(0,4) node[above] {$X_t$};
\foreach \y/\ytext in {0.5/1, 1/2, 1.5/3,3.5/p}
\draw[shift={(0,\y)}] (-2pt,0pt) -- (2pt,0pt) node[left] {$\ytext$};

\draw[color=blue,very thick,domain=0:1.90] plot (\x,{1.5});
\draw[color=blue,very thick,domain=2.1:4.4] plot (\x,{0.5});
\draw[color=blue,very thick,domain=4.5:7.9] plot (\x,{3.5});
\draw[color=blue,very thick,domain=8.1:10.9] plot (\x,{0.5});
\draw[color=blue,very thick,domain=11.1:13] plot (\x,{1.0});

\draw[color=blue] (2,1.5) circle (3pt); 
\draw[color=blue,fill] (2,0.5) circle (3pt); 
\draw[color=blue] (4.5,0.5) circle (3pt); 
\draw[color=blue,fill] (4.5,3.5) circle (3pt); 
\draw[color=blue] (8.0,3.5) circle (3pt); 
\draw[color=blue,fill] (8,0.5) circle (3pt); 
\draw[color=blue] (11.0,0.5) circle (3pt); 
\draw[color=blue,fill] (11.0,1.0) circle (3pt);

\draw[thick, snake=brace,segment aspect=0.5] (0,1.7) -- (2,1.7);
\draw[thick] (1.5,2.0) node[above,rotate=0] {$\sim {f_{3}}$};

\draw[thick, snake=brace,segment aspect=0.5] (2,0.7) -- (4.5,0.7);
\draw[thick] (3.75,1.0) node[above,rotate=0] {$\sim {
    f_{1}}$};

\draw[thick, snake=brace,segment aspect=0.5] (4.5,3.7) -- (8,3.7);
\draw[thick] (6.8,4.0) node[above,rotate=0] {$\sim {
    f_{p}}$};

\draw[thick, snake=brace,segment aspect=0.5] (8,0.7) -- (11,0.7);
\draw[thick] (9.9,1.0) node[above,rotate=0] {$\sim {
    f_{1}}$};

\foreach \x/\xtext in {2.0/S_1,4.5/S_2,8/S_3,11/S_4}
\draw[shift={(\x,0)}] (0pt,2pt) -- (0pt,-2pt) node[below] {$\xtext$};
\end{tikzpicture}
\caption{\label{fig:X-process} Construction of a semi--Markov process based on Mittag--Leffler distributed interarrivals.}
\end{figure}
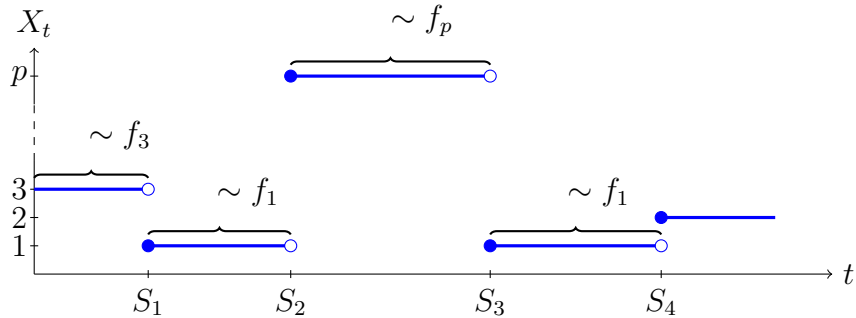

Define the intensity matrix $\mat{\Lambda}=\{ \lambda_{ij}\}_{i=1,...,p}$ by 
\[  \lambda_{ij}=\lambda_i q_{ij}, \ i\neq j, \ \ \mbox{and} \ \ \lambda_{ii}=-\lambda_i = \sum_{k\neq i}\lambda_{ik}  ,  \]
and let 
\[  p_{ij}(t)  = \Prob (X_t=j | X_0=i), \ \ \mat{P}(t) = \{ p_{ij}(t) \}_{i,j=1,...,p} .\]
\begin{theorem}We have
\[   \mat{P}(t) = E_{\alpha,1}(\mat{\Lambda}t^\alpha) .  \] 
\end{theorem}
\begin{proof}
Conditioning on the time of the first jump, we get that
\begin{eqnarray*}
 p_{ij}(t)&=&\delta_{ij} \Prob (T^i_1>t) + \int_0^t f_{i}(s)\sum_{k\neq i} q_{ik} p_{kj}(t-s)\dd s \\
 &=& \delta_{ij}E_{\alpha,1}(-\lambda_i t^{\alpha}) + \sum_{k\neq i} q_{ik} \int_0^t 
 \lambda_i s^{\alpha-1}E_{\alpha,\alpha}(-\lambda_i s^{\alpha}) p_{kj}(t-s)\dd s .
 \end{eqnarray*} 
 Taking Laplace transforms, and using that 
 \[  {\mathcal L}\left[x^{\beta-1} E_{\alpha, \beta}\left(a x^{\alpha}\right)\right](s)=s^{-\beta}\left(1-a s^{-\alpha}\right)^{-1} , \]
together with $\lambda_{ik}=\lambda_i q_{ik}$, we get that
\begin{eqnarray*}
\hat{p}_{ij}(s):= {\mathbb E}(e^{-s p_{ij}(t)})&=& \delta_{ij} \frac{1}{1+\lambda_i s^{-\alpha}} + \sum_{k\neq i} \lambda_{ik} 
\frac{s^{-\alpha}}{1+\lambda_i s^{-\alpha}} \cdot \hat{p}_{kj}(s)
\end{eqnarray*}
or
\begin{eqnarray*}
(1+\lambda_i s^{-\alpha})\hat{p}_{ij}(s)&=& \delta_{ij} + s^{-\alpha}\sum_{k\neq i} \lambda_{ik} 
 \hat{p}_{kj}(s) . 
\end{eqnarray*}
Now using $\lambda_{ii}=-\lambda_i$, we get
\begin{equation}
   \hat{p}_{ij}(s) = \delta_{ij} + s^{-\alpha} \sum_{k=1}^p \lambda_{ik}\hat{p}_{kj}(s) .
\label{eq:thesame}
    \end{equation} 
In matrix form this amounts to
\[  \hat{\mat{P}}(s) = \mat{I} + s^{-\alpha} \mat{\Lambda}\hat{\mat{P}}(s) \]
which has the solution
\[  \hat{\mat{P}}(s) = (\mat{I}-s^{-\alpha}\mat{\Lambda})^{-1} .  \]
The right-hand side is the Laplace transform of ${E}_{\alpha,1}(\mat{\Lambda}t^\alpha)$, establishing the result.
\end{proof}
Next we consider the case where $E=\{1,2,...,p,p+1\}$ and where the states $1,...,p$ are transient and state $p+1$ is absorbing (with respect to the Markov chain $\{ Y_n\}_{n\in\mathbb{N}}$). This means that $\{ Y_n\}_{n\in\mathbb{N}}$ has a transition matrix of the form
\[   \mat{Q} = \begin{pmatrix}
\mat{Q}^1 & \vect{q}^1 \\
\vect{0} & 1 
\end{pmatrix} ,  \]
and regarding the intensities we set $\lambda_{p+1}=0$. The matrix $\mat{\Lambda}$ then is of the form 
\begin{equation}
 \mat{\Lambda} 
=
\begin{pmatrix}
\mat{T} & \vect{t} \\
\vect{0} & 0 
\end{pmatrix} . \label{eq:PH-structure-matrix}
\end{equation}
   
We notice the following useful result.   
\begin{lemma}
 \[  E_{\alpha,1}\left(\begin{pmatrix}
\mat{T} & \vect{t} \\
\vect{0} & 0 
\end{pmatrix} x^\alpha \right) = 
\begin{pmatrix}
E_{\alpha,1} (\mat{T}x^{\alpha}) & \vect{e}-E_{\alpha,1} (\mat{T}x^{\alpha})\vect{e} \\
\vect{0} & 1 
\end{pmatrix}  ,
  \]
  where $\vect{t}=-\mat{T}\vect{e}$ (rows sum to zero).
\end{lemma}   
\begin{proof}
By definition,
\begin{eqnarray*}
 E_{\alpha,1}\left(\begin{pmatrix}
\mat{T} & \vect{t} \\
\vect{0} & 0 
\end{pmatrix} x^\alpha \right)&=& 
\sum_{n=0}^\infty \begin{pmatrix}
\mat{T} & \vect{t} \\
\vect{0} & 0 
\end{pmatrix}^n \frac{x^{\alpha n}}{\Gamma (\alpha n+1)} \\
&=&\mat{I} + \sum_{n=1}^\infty \begin{pmatrix}
\mat{T} & \vect{t} \\
\vect{0} & 0 
\end{pmatrix}^n \frac{x^{\alpha n}}{\Gamma (\alpha n+1)} \\
&=&\mat{I} + \sum_{n=1}^\infty \begin{pmatrix}
\mat{T}^n & -\mat{T}^n\vect{e} \\
\vect{0} & 0 
\end{pmatrix} \frac{x^{\alpha n}}{\Gamma (\alpha n+1)} \\
&=& \begin{pmatrix}
\mat{I}+\sum_{n=1}^\infty \mat{T}^n \frac{x^{\alpha n}}{\Gamma (\alpha n +1)} & -\left(\sum_{n=1}^\infty \mat{T}^n  \frac{x^{\alpha n}}{\Gamma (\alpha n +1)}\vect{e}\right) \\
\vect{0} & 1 
\end{pmatrix} \\
&=& \begin{pmatrix}
E_{\alpha,1} (\mat{T}x^{\alpha}) & \vect{e}-E_{\alpha,1} (\mat{T}x^{\alpha})\vect{e} \\
\vect{0} & 1 
\end{pmatrix}.
\end{eqnarray*}
\end{proof}
Thus the restriction of $E_{\alpha,1}(\mat{\Lambda}x^\alpha)$ to the transient states $1,...,p$ equals $E_{\alpha,1}(\mat{T}x^\alpha)$ and is hence the sub--transition matrix between the transient states.
\begin{theorem}\label{Cor:MML-of-sample-path-model}
Let $\{X_t\}_{t\ge 0}$ be a semi-Markov process, where the matrix $\mat{\Lambda}$ has the form
\[ \mat{\Lambda} 
=
\begin{pmatrix}
\mat{T} & \vect{t} \\
\vect{0} & 0 
\end{pmatrix} 
  . \] 
Let $\tau=\inf\{t\ge 0:\, X_t= p+1\}$ denote the time until absorption. 
Then $\tau$ has a MML($\alpha,\vect{\pi},\mat{T}$) distribution, with cumulative distribution function given by
$$
F_\tau(u)=1-\vect{\pi} E_{\alpha,1}(\mat{T} u^\alpha) \vect{e} .
$$
\end{theorem}
\begin{proof}
For \(E^\ast=\{1,2, \ldots, p\}\), the events \(\{\tau>u\}\) and \(\left\{X_{u} \in E^\ast\right\}\) coincide. Thus,
\begin{align*}
1-F_\tau(u)&=\mathbb{P}(\tau>u)=\Prob\left(X_{u} \in E^\ast\right)=\sum_{j=1}^{p} \Prob\left(X_{u}=j\right)\\
&=\sum_{i, j=1}^{p}\Prob\left(X_{u}=j | X_{0}=i\right) \Prob\left(X_{0}=i\right)
=\sum_{i, j=1}^{p} \pi_{i} \boldsymbol{P}_{i j}(u)=\vect{\pi} E_{\alpha,1}(\mat{T} u^\alpha) \vect{e}.
\end{align*}

\end{proof}

{\begin{remark} \normalfont
The proofs above heavily depend on the form of the Laplace transform of the ML distribution, and its similarity with the exponential function. Note that the above construction naturally extends the definition of PH distributions as absorption times of continuous--time Markov chains, the latter being the limit case $\alpha\to 1$. In general, such a semi-Markov representation will hence not be available for other product distributions.
\end{remark}}
}

\section{Statistical modeling using MML distributions}\label{secillu}
In this section we present some examples of MML distribution fitting to data. Let us start with an illustration of the maximum likelihood  fitting performance to simulated data. 
\begin{example}\normalfont
We simulate  $300$ observations from $X\sim\mbox{MML}(\alpha,\vect{\pi},\mat{T})$, with a mixture of Erlang PH component, with parameters chosen in such a way that the log-data is trimodal. The corresponding maximum likelihood fit is depicted in Figure \ref{simfit} (for visualization purposes the scale of the $x$-axis is logarithmic). The true parameters are $m=3$, $\alpha=0.9$, $p_1=p_2=p_3=3$, $\lambda_1=10,\:\lambda_2=1,\:\lambda_3=0.1$,  $\theta_1=0.3,\:\theta_2=0.3\:\theta_3=0.4$, whereas the maximum likelihood estimator is found to be
\begin{align*}
\hat\alpha=0.905532 ,\quad \hat \theta_1=0.3133867 ,\quad \hat \theta_2=0.3138739 ,\quad \hat \theta_3=0.3727393,\\
\quad \hat\lambda_1= 9.193643,\quad \hat\lambda_2=1.137208 ,\quad \hat\lambda_3=0.08746225. 
\end{align*}
The negative log-likelihood at the fitted parameters was $1020.102$, compared to $1023.963$ at the true parameters (i.e., in the likelihood sense, the fitted model even outperforms the true model for the simulated data points). Note that the tri-modal shape of the underlying density is nicely identified here (which in pure PH fitting would typically not work as smoothly). We also provide a Hill plot of the simulated data, where the potentially heavy-tailed behavior can be clearly appreciated.\hfill $\Box$

\begin{figure}[hh]
\centering
\includegraphics[width=7.3cm,trim=.5cm .5cm .5cm .5cm,clip]{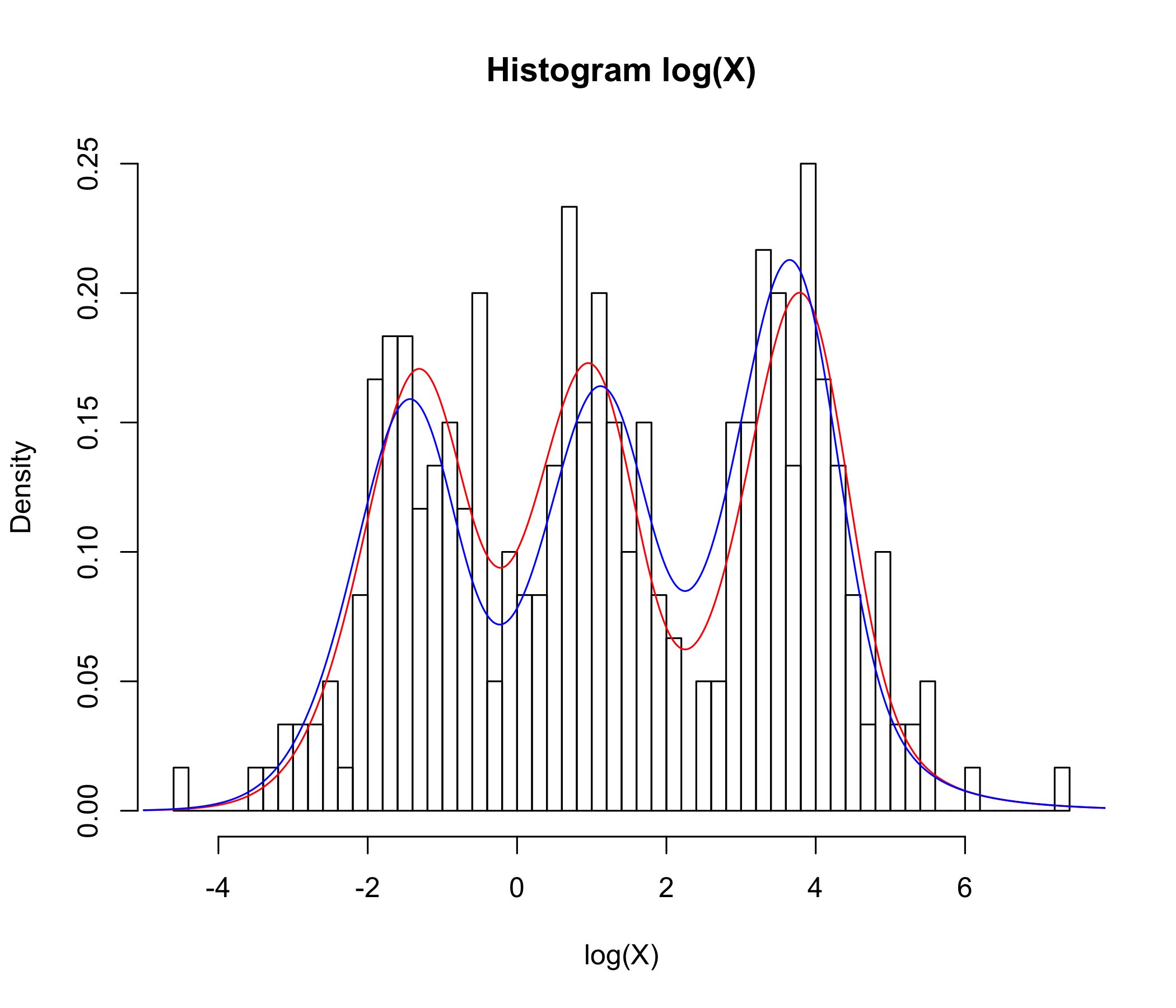}
\includegraphics[width=7.3cm,trim=.5cm .5cm .5cm .5cm,clip]{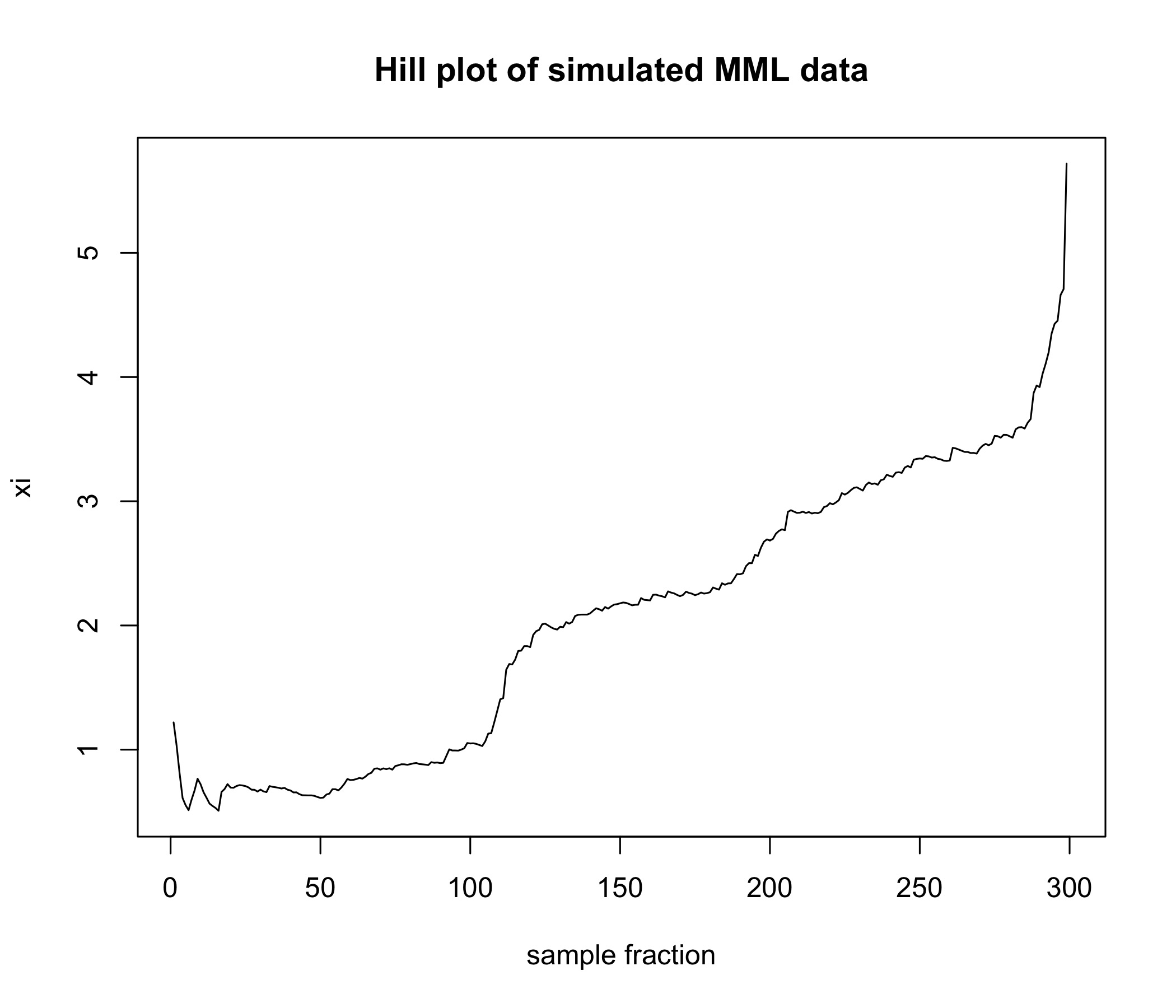}
\caption{Left panel: maximum likelihood fit (red) to simulated MML data with mixture of Erlang PH component and parameters $m=3$, $\alpha=0.9$, $p_1=p_2=p_3=3$, $\lambda_1=10,\:\lambda_2=1,\:\lambda_3=0.1$,  $\theta_1=0.3,\:\theta_2=0.3,\:\theta_3=0.4$. The true density is plotted in blue. Right panel: Hill plot of the untransformed data.} 
\label{simfit}
\end{figure}

\end{example}

The fact that MML distributions behave in a Pareto manner with parameter $\alpha\in (0,1]$ in the tail can be seen from \eqref{prodrepresentation}. Since such a tail will be too heavy for most applications, we introduce a simple power transformation to gain flexibility for the tail behavior, which particularly allows lighter tails as well. 

\begin{definition}
	Let $X\sim \mbox{MML}(\alpha,\vect{\pi},\mat{T})$. For $\nu>0$, we define 
	\[ X^{1/\nu}\sim \mbox{PMML}(\alpha,\vect{\pi},\mat{T},\nu), \]
	and refer to it as the class of Power-MML (PMML) distributions. 
	%
\end{definition}
The density function of a $\mbox{PMML}(\alpha,\vect{\pi},\mat{T},\nu)$ distribution is given by 
\[  f(x)=\nu x^{\nu\alpha-1}\vect{\pi}E_{\alpha,\alpha}\left(\mat{T}x^{\nu\alpha} \right)\vect{t}, \]
which will be needed for the maximum likelihood procedure below. 

{ 
\begin{remark}\normalfont 
The introduction of the PMML class allows for an adaptive transformation of the data during the fitting procedure. The interpretation of $\nu$ is then as the power to which the data should be taken in order for the latter to be most adequately fit by a pure MML distribution. As the number of MML components grows, the product $\alpha \nu$ is expected to estimate the tail index. However, this estimate might be far off when the matrix $\mat{T}$ is not large enough in order for the global fit to be adequate. In those cases, the power transform will tend to improve the fit of the body of the distribution, rather than the tail. When compared to the approach taken in \cite{ab18inh} (fitting a PH density to log-transformed heavy tailed variables) one  can consider the present procedure as adaptive selection of the transformation function, as opposed to fixing it to be the logarithm.
\end{remark}
}
\begin{example}\normalfont
We consider a real-life motor third party liability (MTPL) insurance data set which was thoroughly studied in \cite{abt}, mainly from a heavy-tailed perspective (referred to as "Company A" there). The data set originally consists of $837$ observations, having the interpretation of claim sizes reported to the company during the time frame 1995-2010. The data are right-censored, and were analyzed recently in \cite{abb} using perturbed likelihood with censoring techniques. For the present purpose, we solely focus on the \textit{ultimates}, which consists of imputing an expert prediction of the final claim amount for all claims which are still open, i.e.\ right-censored. We restrict our analysis here to the largest $800$ observations, since the inclusion of the 37 smallest claims lead to a sub-optimal fit, but are somehow irrelevant for modeling purposes. 
For convenience, we divided the claim sizes by $100,000$.

The heavy-tailed nature of the data suggests that using MML distributions to model the claim sizes is appropriate. Recently, in \cite{abbtrim}, a tail index of $\alpha^{-1}=0.48$ was suggested through an automated threshold selection procedure, using a novel trimming approach for the Hill estimator. Since this (or also other much rougher pre-analysis techniques like Pareto QQ-plots) suggests a finite mean, we employ the PMML distributions for the present purpose. This is in fact advised as a general procedure, since in situations where a pure MML  fit is appropriate, the fitting procedure will suggest a value for $\nu$ close to 1 anyway. The maximum likelihood procedure identifies here a surprisingly simple PMML distribution as adequate, namely with a PH component just being a simple exponential random variable:
\begin{align*}
\hat\alpha=0.3025553,\quad \hat{\vect{T}}=-0.08293046,\quad \hat\nu=6.941576.
\end{align*}
More complex PH components turn out to indeed numerically degenerate into this simple model again. The resulting model density is hence given by 
\[ f(x)=0.56\, x^{1.1}E_{0.30,0.30}\left(-0.08 x^{2.1} \right).\]
 The adequate fit can be appreciated in Figure \ref{ultimates}. Observe that the maximum likelihood approach is concerned not only with the tail behaviour but also with adequately fitting the body of the distribution. Nonetheless, the tail index of the PMML fit is given by $(\hat\alpha\cdot\hat\nu)^{-1}=0.4761427$, which is strikingly(!) close to the $0.48$ suggested in \cite{abbtrim}. 

\noindent A previous approach to describe the entire data set by one model was given in \cite[p.99]{abt}, where a splicing point was suggested for this data set at around the $20$th largest order statistic, based purely on expert opinion. A semi-automated approach in \cite{abbtrim} suggested splicing at the $14$th largest data point. Notice that not only does our model fit the data well and is much more parsimonious, but it also circumvents threshold or splicing point selection completely. \hfill $\Box$

\begin{figure}[]
\centering
\includegraphics[width=7.3cm,trim=.5cm .5cm .5cm .5cm,clip]{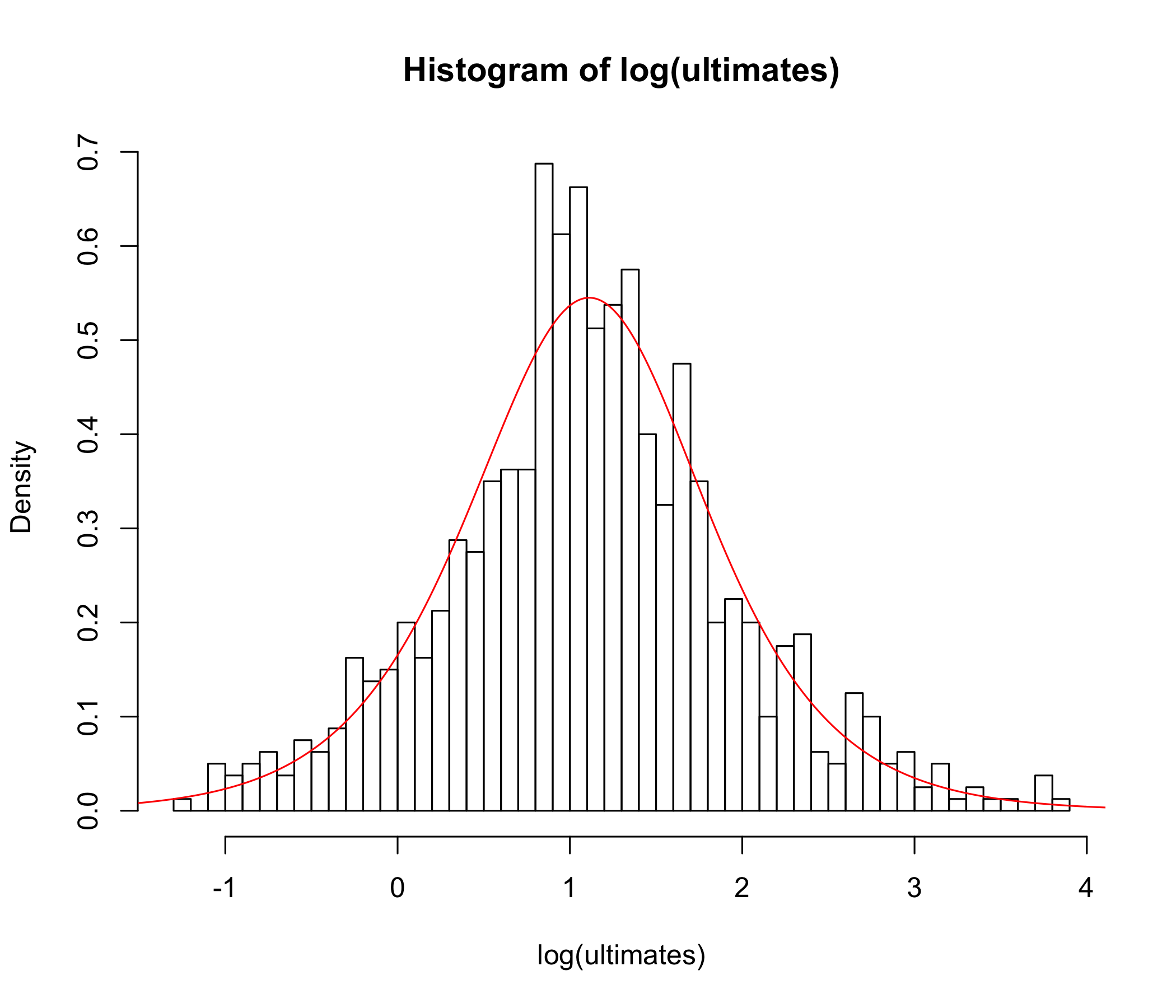}
\includegraphics[width=7.3cm,trim=.5cm .5cm .5cm .5cm,clip]{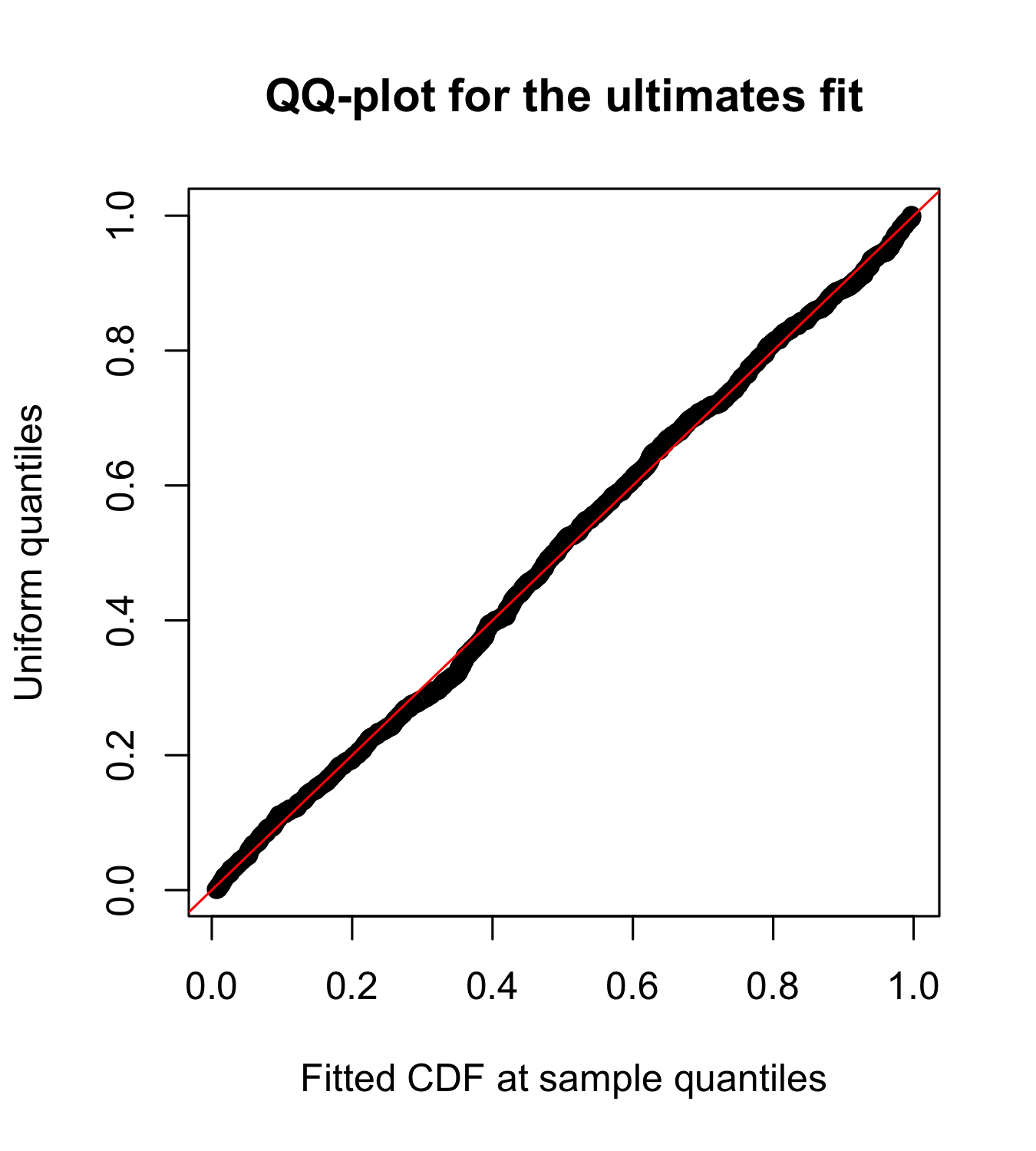}
\caption{Liability insurance ultimates. Left panel: maximum likelihood fit using a PMML with exponential PH component. Right panel: QQ-plot of the fitted distribution function evaluated at the sample quantiles, against theoretical uniform quantiles.} 
\label{ultimates}
\end{figure}

\end{example}

Phase-type distributions are weakly dense in the set of all probability distributions on the positive real line. However, often a very large dimension of the PH distribution is needed to get a decent fit to data. Here, we show an example of how the class of PMML distributions can be used to reduce the dimension of a PH fit, thanks to the increased flexibility that the randomization with an $\alpha$-stable distribution and the power function $(\cdot)^{1/\nu}$ provide.
\begin{example}\normalfont
We consider $n=500$ simulated data points $X_1,\dots,X_{n}$ following a mixture of two Erlang(40) distributions. The PH representation is of dimension $80$, with parameters
\begin{align*}
\pi^0_1=\pi^0_{41}=0.5, \quad \lambda^0_1=100,\quad \lambda^0_2=50.
\end{align*}
Since the data is light-tailed, and PMML are heavy-tailed, we consider the transformed observations
\begin{align}
Y_i=\exp(X_i)-1, \quad i=1,\dots,n,\label{exptrafo}
\end{align}
which are Pareto in the tail. We then proceed to fit a PMML distribution to the transformed data, but with a much lower matrix dimension. Concretely, we consider a mixture of two Erlang distributions of three phases each for the PH component of the PMML representation. In this way we are led to the maximum likelihood estimates
\begin{align*}
\hat\alpha=0.8649503,\quad \hat\pi_1=0.5386982,\quad \hat\pi_1=0.4613018,\\
\quad \hat\lambda_1=25.47413,\quad \hat\lambda_2=1.298168,\quad \hat\nu=3.871273.
\end{align*}
The (back-transformed) fitted density is plotted in Figure \ref{expphfit}, along with a histogram of the original PH data points. We also include a fitted density using a pure PH distribution of the same dimension and kind: a mixture of two Erlangs of three phases each. We observe how transforming the data into the heavy-tail domain, fitting a PMML and then back-transforming adds only two extra parameters and improves the estimation dramatically. Additionally, a Hill plot of the transformed data is provided, which shows that the tail index empirically could correspond to $\alpha^{-1}\approx 0.1$, such that it is necessary to use the PMML class, as opposed to only the MML. For reference, the resulting tail is $(\hat\alpha\cdot\hat\nu)^{-1}=0.223427$, but here the quantification of the tail behaviour is not the main focus of the estimation, and used only qualitatively. {  In fact, a quick calculation shows that the true index is $\alpha^{-1}=1/50=0.02$, and it is well-known that it is a hard task to estimate tail indices in the transition area between Fr\'echet and Gumbel domains of attraction.}  The additional Hill plots of simulated paths of the estimated model in Figure \ref{expphfit} show how the hump at the middle of the Hill plot is not a random fluctuation, but rather a systematic feature.

{  Notice that we took the exponential transformation \eqref{exptrafo} of the simulated data because in this case we know that their tail is exponential and hence the transforms will be regularly varying, in accordance with the PMML distribution. In general, the exponentiation of light-tailed data does not imply that the resulting underlying distribution is regularly varying in the tail. Hence, for real data, a preliminary assessment of the tail behavior and the one of their exponential transforms is recommended. }

\begin{figure}[hh]
\centering
\includegraphics[width=7.3cm,trim=.5cm .5cm .5cm .5cm,clip]{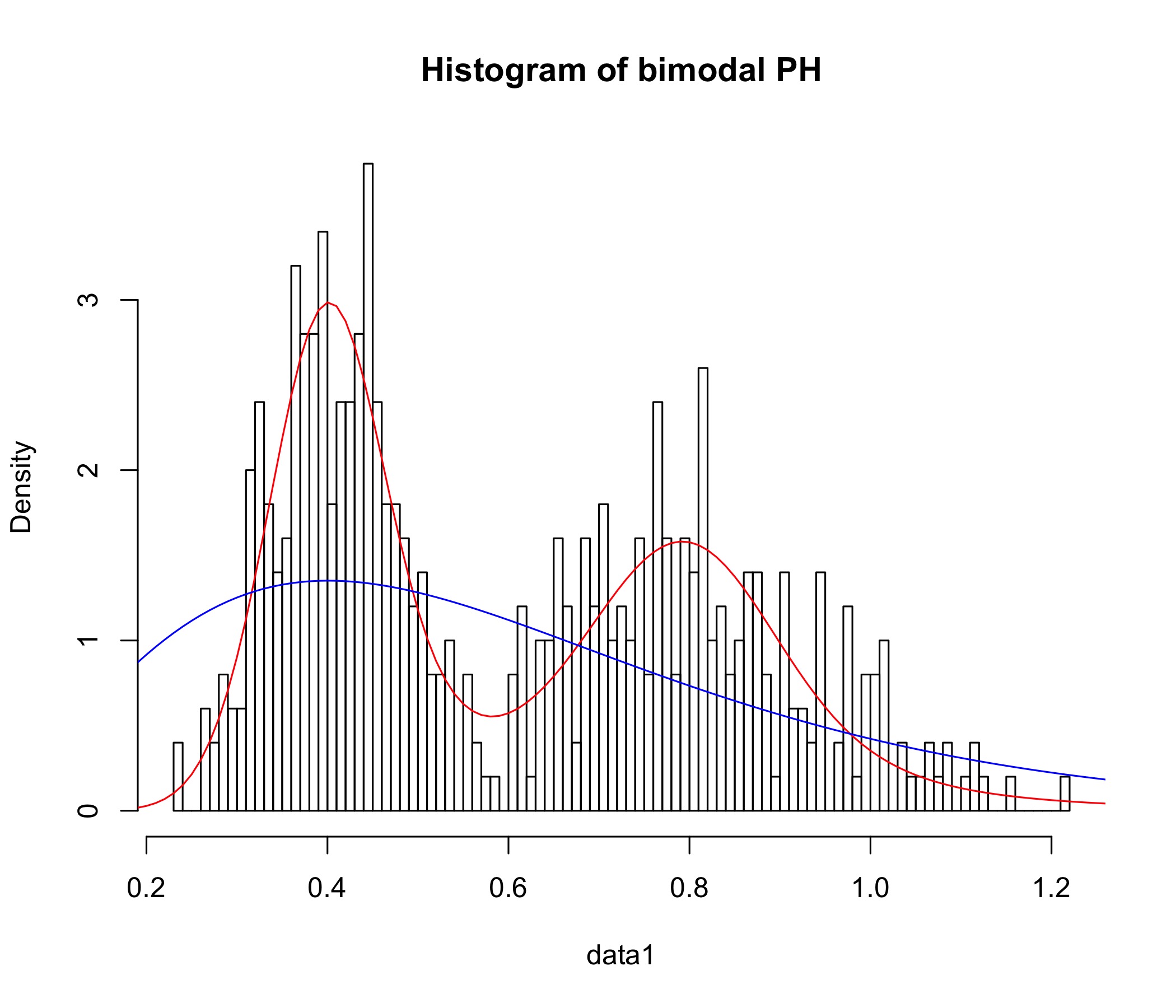}
\includegraphics[width=7.3cm,trim=.5cm .5cm .5cm .5cm,clip]{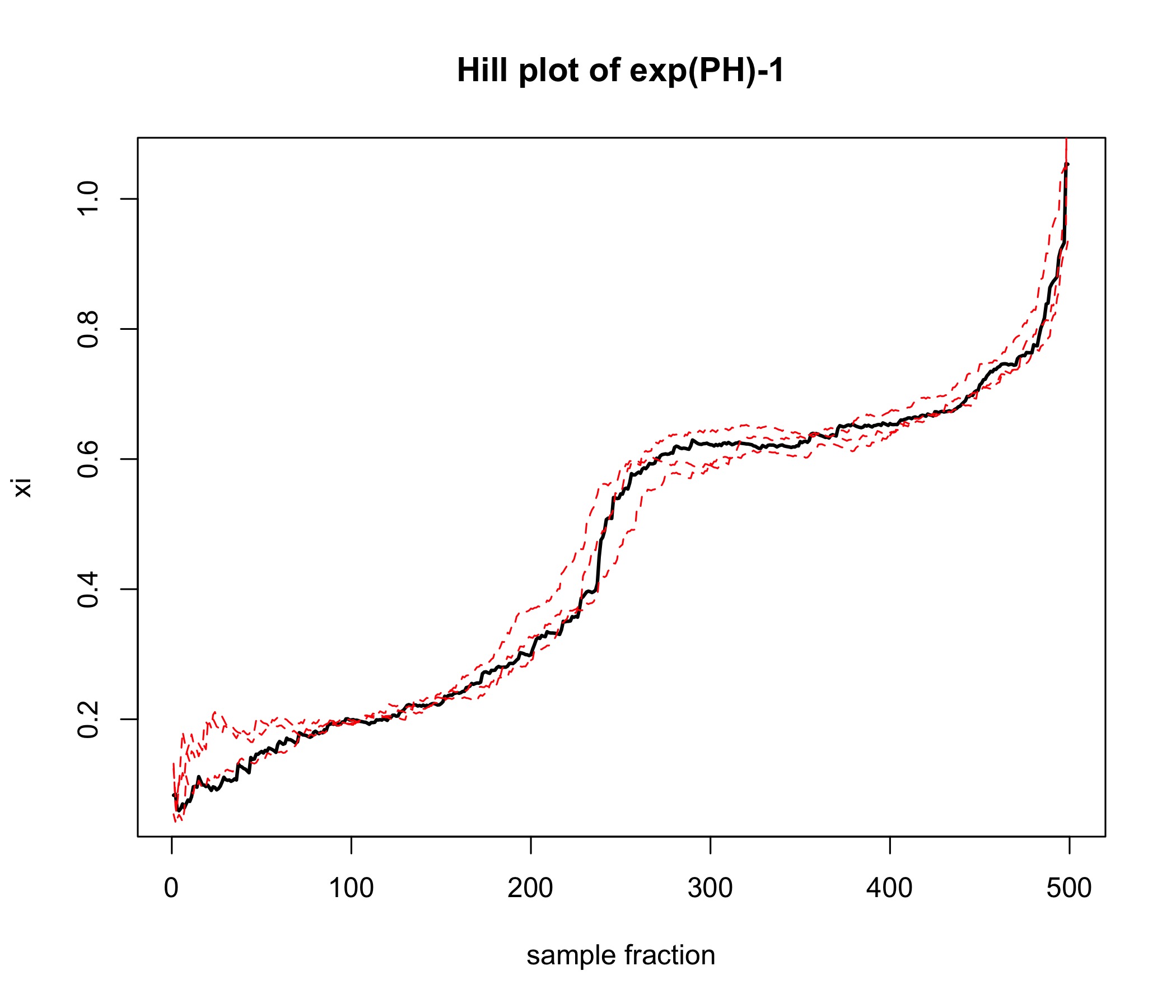}
\caption{Left panel: a back-transformed PMML fit (red) using a $6$-dimensional matrix representation, and a pure PH fit (blue) of the same dimension, to an $80$-dimensional PH simulated data set. Right panel: Hill plot of the transformed data (black, solid), together with Hill plots for simulated data from the resulting estimated model (red, dashed).} 
\label{expphfit}
\end{figure}

\end{example}

{ 
\begin{remark}\rm
When fitting a MML or PMML distribution to data, one has to decide upon the dimension of the underlying phase--type representation. This problem arises similarly when fitting phase--type distributions to light-tailed data, and there are no generally accepted and well established methods for model selection, since PH distributions may be well overparametrised so that penalized methods such AIC or BIC indices will not work in general. 
The order of the PMML (or phase--type distribution) is therefore usually chosen by fitting a range of models of different dimension and then comparing the fit and likelihood values (which, as opposed to information indices, are comparable). 
\end{remark}}
\section{Conclusion}\label{secconcl}
In this paper we define the class of matrix Mittag-Leffler distributions and derive some of its properties. We identify this class as a particular case of inhomogeneous phase-type distributions under random scaling with a stable law, which together with its power transforms is surprisingly versatile for modeling purposes. {  In addition, the class is shown to correspond to absorption times of semi--Markov processes with Mittag--Leffler distributed interarrival times, providing a natural extension of the phase--type construction.} We illustrate with several examples  that this class can simultaneously fit the main body and the tail of a distribution with remarkable accuracy in a parsimonious manner. It turns out that the flexibility of this heavy-tailed class of distributions can even make it worthwhile to transform data into the heavy-tailed domain, fitting the resulting data points and then transforming them back. It will be an interesting direction for future research to further explore the potential of such fitting procedures, both from a theoretical and practical perspective. \\

\textbf{Acknowledgement.} H.A. acknowledges financial support from the Swiss National Science Foundation Project 200021\_168993.

\bibliographystyle{plain}
\bibliography{HMM_rev.bib}

\end{document}